\documentclass[review,hidelinks,onefignum,onetabnum]{siamart220329}

\usepackage{lipsum}
\usepackage{amsfonts}
\usepackage{graphicx}
\usepackage{epstopdf}
\usepackage{algorithmic}
\ifpdf
  \DeclareGraphicsExtensions{.eps,.pdf,.png,.jpg}
\else
  \DeclareGraphicsExtensions{.eps}
\fi

\newsiamremark{remark}{Remark}
\newsiamremark{hypothesis}{Hypothesis}
\crefname{hypothesis}{Hypothesis}{Hypotheses}
\newsiamthm{claim}{Claim}

\headers{Boundary treatment for IMEX Runge-Kutta schemes}{V. González, J.G. López, M.J. Castro, and J.A. García}

\title{Boundary treatment for high-order IMEX Runge-Kutta local discontinuous Galerkin schemes for multidimensional nonlinear parabolic PDEs\thanks{Submitted to the editors DATE.
\funding{The third author's research has been funded by FEDER and the Spanish Government through the coordinated Research project RTI2018-096064-B-C1. Also, this research has been partially funded by MCIN/AEI/10.13039/501100011033 and by the “European Union NextGenerationEU/PRTR” through the grant PDC2022-133663-C21 and by MCIN/AEI/10.13039/50110001103 and by “ERDF A way of making Europe”, by the “European Union” through the grant PID2022-137637NB-C21.
The other authors' research has been funded by the Spanish MINECO under research project number PID2022-141058OB-I00 and by the grant ED431G 2019/01 of CITIC, funded by Conseller\'ia de Educaci\'on, Universidade e Formaci\'on Profesional of Xunta de Galicia and FEDER.}}}

\author{V. González-Tabernero\thanks{Department of Mathematics and CITIC, Facultad de Informática, University of A Coruña, Campus de Elviña s/n, A Coruña, 15071, Galicia, Spain(\email{v.gonzalez.tabernero@udc.es}, \email{jose.lsalas@udc.es},  \email{jose.garcia.rodriguez@udc.es}).}
\and J. G. López-Salas\footnotemark[2]
\and M. J. Castro-Díaz\thanks{Department of An\'alisis Matem\'atico, Facultad de Ciencias, University of M\'alaga, Campus de Teatinos s/n, M\'alaga, 29080, Andalucía, Spain  (\email{mjcastro@uma.es}).}
\and J. A. García-Rodríguez\footnotemark[2]
}

\usepackage{amsopn}

\renewcommand{\l}{\left(}
\renewcommand{\r}{\right)}

\newcommand{\R}{\mathbb{R}}

\newcommand{\mb}[1]{\mathbf{#1}}
\newcommand\mbx{\mb{x}}

\newcommand\mbF{\mb{F}}
\newcommand\mbG{\mb{G}}

\newcommand\mbq{\mb{q}}

\usepackage{stix}
\newcommand*{\greysquare}{\textcolor{gray}{\squareneswfill}}
\renewcommand\div[1]{{\rm div\,} {#1}}

\usepackage{tikz}
\usepackage{subfigure}

\usepackage{enumitem}

\ifpdf
\hypersetup{
  pdftitle={Boundary treatment for IMEX Runge-Kutta schemes},
  pdfauthor={V. González, J.G. López, M.J. Castro, and J.A. García}
}
\fi

\nolinenumbers
\begin{document}

\maketitle

\begin{abstract}
In this article, we propose novel boundary treatment algorithms to avoid order reduction when implicit-explicit Runge-Kutta time discretization is used for solving convection-diffusion-reaction problems with time-dependent Di\-richlet boundary conditions. We consider Cartesian meshes and PDEs with stiff terms coming from the diffusive parts of the PDE. 
The algorithms treat boundary values at the implicit-explicit internal stages in the same way as the interior points. The boundary treatment strategy is designed to work with multidimensional problems with possible nonlinear advection and source terms.
The proposed methods recover the designed order of convergence by numerical verification. For the spatial discretization, in this work, we consider Local Discontinuous Galerkin methods, although the developed boundary treatment algorithms can operate with other discretization schemes in space, such as Finite Differences, Finite Elements or Finite Volumes.
\end{abstract}



\section{Introduction}

The goal of this work is to develop high-order implicit-explicit numerical schemes for solving the following nonlinear multidimensional time-dependent scalar convection-diffusion-reaction PDE with time-dependent Dirichlet boundary conditions
\begin{align}
    &u_t+\div{\mbF}(u) = \div(\mbG(\nabla u)) + h(u), 
 \quad (\mbx,t)\in\Omega\times(0,T], \label{eq:generalPDE}\\
    &u(\mbx,0) = u_0(\mbx), \quad \mbx \in \Omega, \label{eq:iniCond}\\
    &u(\mbx,t) =\omega(\mbx,t), \quad \mbx \in \Gamma \equiv \partial\Omega, \, t\in(0,T], \label{eq:boundCond}
\end{align}
where $u_0$ and $\omega$ are the initial and boundary conditions, respectively. Besides, the transport flux function is given by ${\mbF}(u) = (f_1(u),\ldots,f_d(u))$ and $h$ defines the reaction terms. Moreover, $\mbG(\nabla u) = (g_1(\nabla u),\ldots,g_d(\nabla u))$ is the function defining the diffusion part. In this work, we consider nonlinear convection and source terms, while, for the sake of simplicity, we only contemplate linear diffusion operators. Finally, $\Omega$ is a bounded rectangular domain in $\R^d$.

For the time discretization of convection-diffusion-reaction equations, when the problem is diffu\-sion-dominated, it is well known in the literature that explicit schemes suffer a severe time step restriction for stability ($\Delta t = \mathcal{O}((\Delta x)^2)$). To overcome this demanding time step restriction, Implicit-Explicit (IMEX) Runge-Kutta (RK) schemes can be applied. More precisely, the convection and source terms are treated explicitly, while the diffusion parts (stiff terms) are handled implicitly. In this way, the stability condition of the IMEX method is of the type $\Delta t = \mathcal{O}(\Delta x)$, i.e., the explicit stability condition of the advection terms of the PDE. In this work, we only deal with non-stiff source terms, although the developed methods could be easily applied to the case of stiff sources, by implicitly treating them.

IMEX-RK methods have been extensively applied in the context of solving stiff ODE systems resulting from the semidiscretization of PDEs: for example, coupled with finite differences and finite volumes in \cite{Russo05}, and coupled with Local Discontinuous Galerkin (LDG) methods in \cite{Wang-Shu-Zhang-2016,Haijin-wang-Shu-M2AN-2016,Haijin-Shu-SIAM-JNA-2016,WANG2018164}.

Although very appealing at first glance, the application of IMEX RK schemes for time-dependent problems has also difficulties. IMEX methods achieve a high order of convergence by considering several intermediate time stages between one discretization time and the following. Generally speaking, the higher the order of the scheme, the higher the number of its internal steps. When IMEX time integrators are applied to solve the semidiscrete version of \eqref{eq:generalPDE}, boundary conditions have to be imposed also at the intermediate RK stages. Problems may appear if these boundary conditions are time-dependent. The naive way of imposing those boundary conditions consists of just the direct evaluation of the function describing the Dirichlet boundary condition at the boundary nodes and at the intermediate times of the IMEX stages. Such a strategy produces a well-known phenomenon of degradation (or loss) of the order of convergence, thus ruining the high-order of the RK scheme, and making its use pointless. 
This phenomenon is known in the literature as order reduction and can be severe. 
More precisely, order reduction may happen when a RK method is used together with the method of lines for the full discretization of an initial boundary value problem, see \cite{ostermannRoche92,sanzSerna86,WANG2018164} and references therein. Imposing a time-dependent boundary condition for the PDE at times $t^{n,i}$ associated with the IMEX stages, i.e. $u^{n,i}=\omega(t^{n,i})$, typically generates bigger errors in space near the boundaries, known as numerical boundary layers in the literature. A common representation of such errors is shown in Figure \ref{fig:BL_motivation}. Those spatial errors near the boundaries of the domain do not generally cancel out, thus yielding a composition of abnormally bigger errors in space in the neighborhood of the boundaries $\Gamma$. These errors greatly reduce the order of convergence of the method. One approach for remedying order reduction consists of modifying the way boundary conditions are imposed, by analyzing the local discretization error of the fully discrete method.
The process of obtaining suitable numerical boundary values for the imposition of the boundary conditions at the intermediate time stages is known as boundary treatment. It is a complex and classic, yet active, field of research. A proper boundary treatment is crucial for retaining the good high-order properties of the RK time integrators in the PDEs context. Otherwise, high-order IMEX schemes applied to stiff problems with time-dependent boundary conditions may be inefficient. Also notice that this phenomenon is not exclusively tied to semi-implicit RK time integrators, like IMEX-RK; but also appears in the case where traditional RK integrators are employed.

\begin{figure}[!htb]
\centering
\subfigure{\includegraphics[scale=0.39]{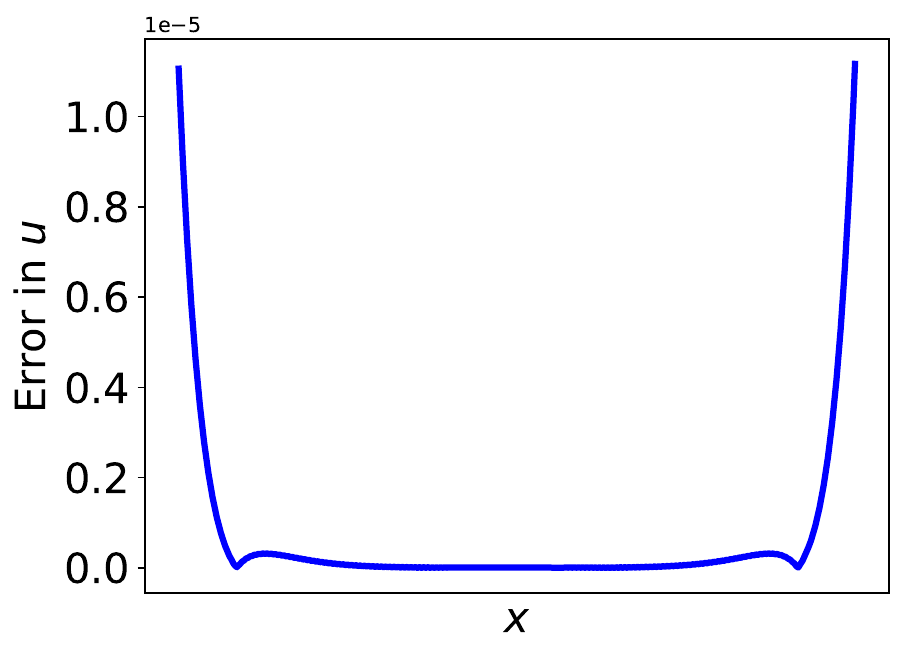}}
\subfigure{\includegraphics[scale=0.39]{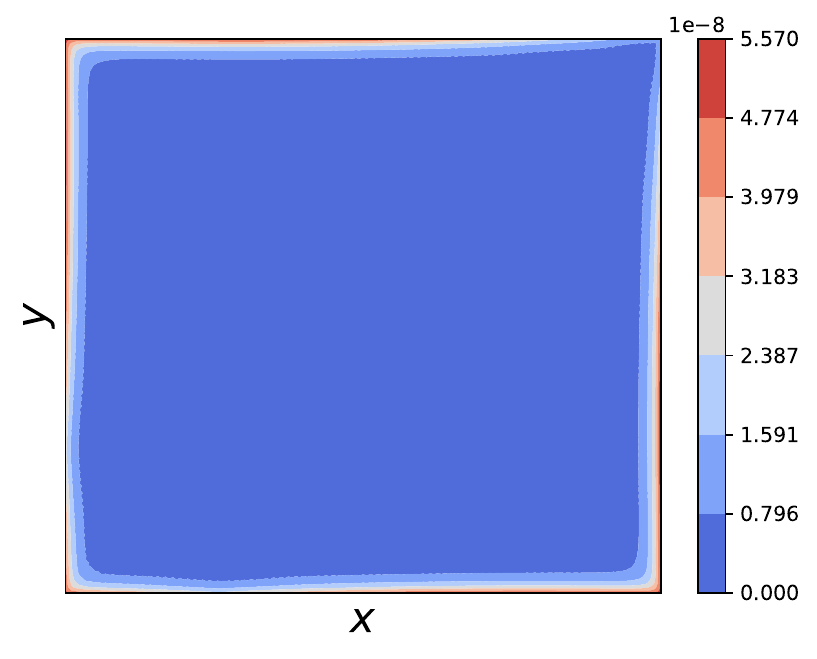}}
\caption{Examples of characteristic errors in space caused by the imposition of conventional boundary conditions on the internal stages of IMEX methods. Left: errors in the one-dimensional problem of Section \ref{subsec:1dheat}. Right: Contour plot of the errors for the two-dimensional problem of Section \ref{sec:exp2d}.}
\label{fig:BL_motivation}
\end{figure}

Boundary treatment strategies for explicit schemes were studied in  \cite{Baeza2016a,doi:10.1137/0719047, LU2016276, TAN20122510,TAN2010,WANG2018164,ZHAO2020,ZHAO2022} and references therein. A classical reference for the boundary treatment of hyperbolic PDEs in finite differences and finite element schemes is \cite{doi:10.1137/0719047}, by  Gottlieb et al. In \cite{Baeza2016a}, Baeza et al. present a technique for the extrapolation of information from the interior of the computational domain to ghost cells designed for structured meshes. The authors applied Lagrange interpolation with a filter for the detection of discontinuities that permits a data-dependent extrapolation, with higher order at smooth regions and essentially non-oscillatory properties near discontinuities. In \cite{LU2016276}, Shu et al. designed an inverse Lax-Wendroff procedure for the imposition of numerical boundary conditions in the convection–diffusion setting with explicit time integrators. In \cite{ZHAO2020}, Zhao et al. present a boundary treatment method for explicit  RK methods for solving hyperbolic systems with source terms. In  \cite{ZHAO2022}, a finite difference boundary treatment method is proposed for RK methods of hyperbolic conservation laws; the method combines an inverse Lax-Wendroff procedure and a WENO-type extrapolation.

For IMEX time integration schemes, to the best of our knowledge, \cite{WANG2018164,ZHAO2020IMEX} are the only works in the literature presenting boundary treatment strategies.  In \cite{ZHAO2020IMEX}, Zhao et al. develop a high-order finite difference boundary treatment method for IMEX RK schemes when solving hyperbolic systems with stiff source terms on structured meshes. The authors compute the solutions at ghost points from the wide stencil of the interior high-order scheme by using the RK scheme at the boundary and an inverse Lax-Wendroff (ILW) procedure. In \cite{WANG2018164}, Shu et al. present a technique to avoid order reduction when third-order IMEX RK time discretization is used together with LDG spatial discretization. The boundary treatment strategy is only \textit{ad hoc} presented for a specific one-dimensional linear convection-diffusion PDE with time-dependent Dirichlet boundary conditions and a particular third-order IMEX scheme. The authors propose a strategy of boundary treatment at each intermediate stage. The general idea is to treat boundary values in the same way as the interior points. A Cauchy-Koval\'eskya procedure (see, for example, \cite{DAKIN2018228,SEAL2017}), in combination with the differentiation of the IMEX equations for the stages, and the numerical approximation of high-order derivatives that appear in the process, is presented. 

Although the technique proposed here can be applied regardless of the employed spatial discretization, we will focus on LDG methods. The LDG method was first presented in \cite{Cockburn-Shu-91-I}. This method has been applied for hyperbolic conservation laws in \cite{Cockburn-Shu-90-IV,Cockburn-Karniadakis-Shu-2000,Cockburn-Shu-89-III,Cockburn-Shu-89-II,Cockburn-Shu-91-I,Zhang-Shu-Wang-2010}, and references therein; and for convection-diffusion-reaction  problems in \cite{Cockburn-Shu-SIAM-JNA-1998,Cockburn-Shu-JSC-2001,Haijin-Shu-2017}, and the references therein. IMEX LDG schemes were studied in \cite{Haijin-wang-Shu-M2AN-2016} and references therein.

The novelties of our work concerning \cite{WANG2018164} are the following: the extension to the nonlinear case and the presence of source terms, the general results for arbitrary order of convergence higher than three, as well as the extension to multidimensional problems. On top of that, a general procedure for boundary treatment is presented here for convection-diffusion-reaction PDEs discretized by the method of lines using general IMEX LDG methods, which simplifies the procedure described in \cite{WANG2018164}. To the best of our knowledge, this is the first time that a general procedure of this type has been proposed for arbitrary IMEX LDG methods. The proposed strategy achieves optimal order of accuracy by numerical verification.

The structure of this paper is the following. We start by making a brief review of IMEX LDG numerical schemes in Section \ref{sec:IMEXLDG}, paying special attention to the alternating numerical fluxes and their boundary conditions implications. In Section \ref{sec:bt_gt} we present the extension of the technique presented in \cite{WANG2018164} to general nonlinear parabolic equations with reaction terms.
In Section \ref{sec:numericalAlgs} we propose novel general algorithms for the boundary treatment of IMEX LDG schemes. In Section \ref{sec:numericalExperiments}, we perform numerical experiments to assess the good performance of the proposed boundary treatments, and we show tables of order of convergence for the empirical validation of the presented techniques. Finally, in the Supplementary Materials \ref{app:A}, \ref{app:B} and \ref{app:C}, we fully detail the application of the boundary treatment technique developed in Section \ref{sec:bt_gt} to the PDE problems of Section \ref{sec:numericalExperiments} (numerical experiments).

\section{IMEX LDG methods} \label{sec:IMEXLDG}
This section discusses the discretization of \eqref{eq:generalPDE}. Firstly, in subsection \ref{subsec:LDG} we describe the space discretization obtained by LDG schemes. Then, subsection \ref{subsec:IMEX} is devoted to IMEX Runge-Kutta schemes applied to the stiff system of differential equations obtained by the previous LDG space discretization.

\subsection{LDG space semidiscretization}\label{subsec:LDG}
In this section, we discuss the spatial semidiscreti\-zation of the problem \eqref{eq:generalPDE}-\eqref{eq:boundCond}.
Here, we follow \cite{Cockburn-Shu-SIAM-JNA-1998}, therefore we start rewriting PDE \eqref{eq:generalPDE} as the following equivalent first-order system of PDEs:
\begin{align}
&u_t+\div{\mbF}(u) = \div \mbG(\mbq(\mbx,t)) + h(u), \quad (\mbx,t)\in\Omega\times(0,T],\label{eq:system1}\\
&\mbq(\mbx,t)= \nabla u(\mbx,t),\label{eq:system2}
\end{align}
with $\mbq(\mbx,t)= (q_1(\mbx,t),\ldots,q_d(\mbx,t))$
and with the same initial condition \eqref{eq:iniCond} and boundary conditions \eqref{eq:boundCond}.  In what follows, for simplicity, we only consider the two-dimensional case ($d=2$), and we use the notation $\mbx=(x,y)$, $\mbF=(f_1,f_2)$ and $\mbG=(g_1,g_2)$.

We assume the rectangular mesh $\hat{\Omega}=\left\{\greysquare_{ij}=[x_i,x_{i+1}]\times[y_j,y_{j+1}] \right\}_{i,j=0}^{N-1,M-1}$ covering the spatial domain $\Omega$ with rectangular elements that will be denoted by $\greysquare_{ij}$. Besides, $\square_{ij}$ represents the boundary of the volume $\greysquare_{ij}$. The area of the $ij$-element is $\Delta_i x  \Delta_j y = (x_{i+1}-x_i)(y_{j+1}-y_j)$, for $i=0,\ldots,N-1$, $j=0,\ldots,M-1$. The north, south, east and west boundaries of the volume $\greysquare_{ij}$ will be denoted as $\Gamma_{ij}^n$, $\Gamma_{ij}^s$, $\Gamma_{ij}^e$, $\Gamma_{ij}^w$, respectively, i.e. $\square_{ij} = \Gamma_{ij}^n \cup \Gamma_{ij}^s \cup \Gamma_{ij}^e \cup \Gamma_{ij}^w$.

Associated to the mesh $\hat{\Omega}$, we define the discontinuous finite element space with tensor product polynomials
$$\hat{V} = \left\{ v \in L^2(\Omega): v|_{\greysquare_{ij}}=v_{ij} \in \mathcal{P}_{k}(\greysquare_{ij}), \forall \greysquare_{ij} \in \hat{\Omega} \right\},$$
where $\mathcal{P}_k(\greysquare_{ij})=\{\text{polynomials in $\greysquare_{ij}$ of degree at most } k \text{ in each variable  } x \text{ and } y\}$. The functions $p$ in this space $\hat{V}$ are allowed to have discontinuities across the boundaries of the volumes. For each boundary of each volume, for any piecewise function $p$, there are two traces along the right-hand (up) and left-hand (down), denoted by $p^+$ and $p^-$, respectively. For the north and south boundaries, right-hand means \textit{from the top} while left-hand refers to \textit{from the bottom}.

Now we detail how to find the numerical solution of the LDG scheme \eqref{eq:system1}-\eqref{eq:system2}. The initial condition $u(x,y,0)\in\hat{V}$ is taken as an approximation of the given initial solution $u_0(x,y,0)$.  For any time $t>0$,  the goal is to find the numerical solution $(u(x,y,t),q_1(x,y,t),$ $q_2(x,y,t))\in \hat{V}\times \hat{V}\times \hat{V}$ satisfying the following variational formulation
\begin{align}
 (u_t,v)_{\greysquare_{ij}} &=\mathcal{C}_{\greysquare_{ij}}(u,v)+\mathcal{H}_{\greysquare_{ij}}(u,v)+\mathcal{D}_{\greysquare_{ij}}(q_1,q_2,v), \label{eq:variational1}\\
(q_1,r_1)_{\greysquare_{ij}} &= \mathcal{K}^1_{\greysquare_{ij}}(u,r_1),\label{eq:variational2}\\
(q_2,r_2)_{\greysquare_{ij}} &= \mathcal{K}^2_{\greysquare_{ij}}(u,r_2),\label{eq:variational3}
\end{align}
in each volume $\greysquare_{ij}$, for any test functions $(v,r_1,r_2)\in\hat{V}\times\hat{V}\times\hat{V}$. Here
\begin{align}
 \mathcal{C}_{\greysquare_{ij}}(u,v)&=({\mbF}(u) ,\nabla v)_{ \greysquare_{ij}}- \langle \tilde{\mbF}(u_{ij}^{in},u_{ij}^{out}) , v \rangle_{\square_{ij}},\label{convective-term-1}\\
 \mathcal{H}_{\greysquare_{ij}}(u,v) &= ( h(u), v )_{\greysquare_{ij}},\\
 \mathcal{D}_{\greysquare_{ij}}(q_1,q_2,v)&=- (\mbG(q_1,q_2), \nabla v)_{\greysquare_{ij}} + \langle \mbG(\tilde{q_1},\tilde{q_2}), v \rangle_{ \square_{ij}}, \label{diffusive-term-1}\\
 \mathcal{K}^1_{\greysquare_{ij}}(u,r_1) &= -(u,(r_1)_x)_{ \greysquare_{ij}}+\langle  \tilde{u} ,r_1 \rangle_{\Gamma_{ij}^e} - \langle  \tilde{u} ,r_1 \rangle_{\Gamma_{ij}^w}\label{diffusive-term-2},\\
\mathcal{K}^2_{\greysquare_{ij}}(u,r_2) &= -(u,(r_2)_y)_{ \greysquare_{ij}}+\langle  \tilde{u} ,r_2 \rangle_{\Gamma_{ij}^n} - \langle  \tilde{u} ,r_2 \rangle_{\Gamma_{ij}^s}\label{diffusive-term-3},
\end{align}
where $u_{ij}^{in}$ and $u_{ij}^{out}$ denote the values of $u$ computed from inside and outside the volume $\greysquare_{ij}$. Besides, $(\cdot,\cdot)_{\greysquare_{ij}}$ and $\langle \cdot,\cdot\rangle_{\Gamma_{ij}}$ are the standard inner products in $L^2(\greysquare_{ij})$ and $L^2(\Gamma_{ij})$, respectively, i.e. the first one represents the following volume integral
$$(u,v)_{\greysquare_{ij}}=\int_{\greysquare_{ij}} u v \,dx dy,$$ while the last one the line integral $$\langle u,v\rangle_{\Gamma_{ij}} = \int_{\Gamma_{ij}} u v\, dl.$$

In equations \eqref{convective-term-1}-\eqref{diffusive-term-3}, terms marked with the tilde command denote a numerical flux. Firstly let us work with the line integral $$\langle\tilde{\mbF}(u_{ij}^{in},u_{ij}^{out}) , v \rangle_{\square_{ij}} = \oint_{\square_{ij}} \tilde{\mbF}(u_{ij}^{in},u_{ij}^{out}) \cdot \mathbf{n} \,v \,dl,$$ where $\mathbf{n}$ is the outward unit normal of the volume boundary $\square_{ij}$.  $\tilde{\mbF}(u_{ij}^{in},u_{ij}^{out}) \cdot \mathbf{n}$ is any monotone numerical flux. In this work we considered the simple local Lax-Friedrichs flux, which is given by $$\tilde{\mbF}(u_{ij}^{in},u_{ij}^{out}) \cdot \mathbf{n} \approx \frac12 \left[ \left( \mbF(u^{in}_{ij}) + \mbF(u^{out}_{ij})\right) \cdot \mathbf{n} - \alpha \left( u^{out}_{ij}-u^{in}_{ij}\right) \right],$$ where $\alpha$ is taken as an upper bound for the eigenvalues of the Jacobian in the $\mathbf{n}$ direction. Secondly, for $\tilde{u}$, $\tilde{q_1}$ and $\tilde{q_2}$ an alternating numerical flux has to be considered (see
 \cite{Cockburn-Shu-SIAM-JNA-1998}).

In our work, the alternating numerical flux at the boundaries of interior volumes $\greysquare_{ij}$, $i=1,\ldots,N-2$, $j=1,\ldots,M-2$ is defined as:
\begin{itemize}
    \item On the east boundary of the volume $\greysquare_{ij}$, $\tilde{u}=u^-_{\Gamma^e_{ij}}$, $\tilde{q}_1={q}_{1_{\Gamma^e_{ij}}}^+$, $\tilde{q}_2=q_{2_{\Gamma^e_{ij}}}^+$.
    \item On the west boundary of the volume $\greysquare_{ij}$,  $\tilde{u}=u^-_{\Gamma_{ij}^{w}}$, $\tilde{q_1}={q_1}^{+}_{\Gamma_{ij}^{w}}$, $\tilde{q_2}={q_2}^{+}_{\Gamma_{ij}^{w}}$.
    \item On the north boundary of the volume $\greysquare_{ij}$, $\tilde{u}=u^-_{\Gamma_{ij}^{n}}$, $\tilde{q_1}={q_1}^{+}_{\Gamma_{ij}^{n}}$, $\tilde{q_2}={q_2}^{+}_{\Gamma_{ij}^{n}}$.  
    \item On the south boundary of the volume $\greysquare_{ij}$, $\tilde{u}=u^-_{\Gamma_{ij}^{s}}$, $\tilde{q_1}={q_1}^{+}_{\Gamma_{ij}^{s}}$, $\tilde{q_2}={q_2}^{+}_{\Gamma_{ij}^{s}}$.
\end{itemize}
The important point is that $\tilde{u}$ and $(\tilde{q_1},\tilde{q_2})$ have to be chosen from different directions. Therefore, selecting $\tilde{u}$ from the right-hand side and $(\tilde{q_1}, \tilde{q_2})$ from the left-hand side is also fine. Nevertheless, in this article, to illustrate the boundary treatment, for the alternating numerical flux, the previous choice is considered, which is also sketched in Figure \ref{fig:alternatingUleftQright}. Then, some numerical fluxes at exterior volume boundaries have to be carefully handled. Firstly, for the volumes $\greysquare_{0j}$, $\tilde{u}$ is set to the Dirichlet boundary condition, i.e. $\tilde{u}_{\Gamma_{0j}}=\omega\lvert_{\Gamma_{0j}}$. Similarly, for the volumes $\greysquare_{i0}$, $\tilde{u}\lvert_{\Gamma_{i0}}=\omega\lvert_{\Gamma_{i0}}$. Both these two impositions are natural under the alternating choice of Figure \ref{fig:alternatingUleftQright}. Next, for the exterior right and up volume boundaries of the volumes $\greysquare_{N-1,j}$ and $\greysquare_{i,M-1}$, numerical fluxes $\tilde{u}$, $\tilde{q_1}$, $\tilde{q_2}$ are firstly inverted, concerning the choice in Figure \ref{fig:alternatingUleftQright}. Therefore $u$ is selected from the right-hand side, while $q_1$ and $q_2$ are taken from the left-hand side. Taking advantage of the same boundary treatment strategy just for $u$ in all $\Gamma(\Omega)$ is the reason behind this change. Such boundary treatment scheme is going to be designed in Section \ref{sec:bt_gt}. If this inversion was not carried out, one would also need to design a proper boundary treatment for Neumann boundary conditions for $q_1$ and $q_2$ at $\Gamma^e(\Omega)$ and $\Gamma^n(\Omega)$. 

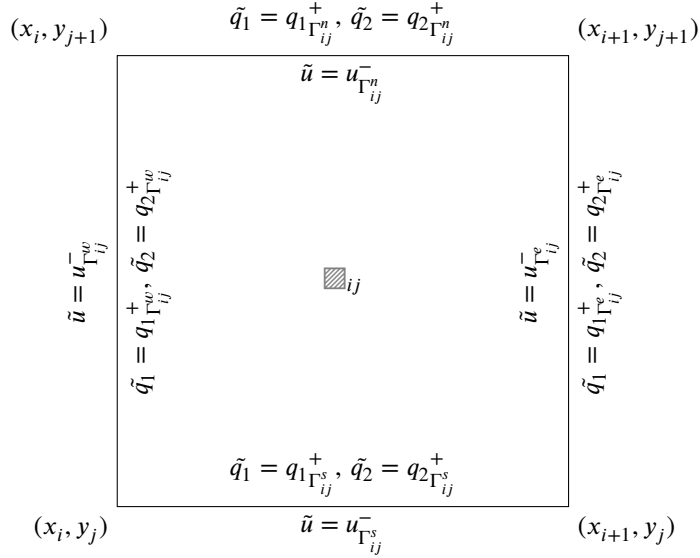
\begin{figure}[!htb]
\begin{center}
\begin{tikzpicture}
  \draw (0, 0) rectangle (6, 6);
  \path
    (0, 0) node[below left] {$(x_i, y_j)$}
    (6, 0) node[below right] {$(x_{i+1}, y_j)$}
    (6, 6) node[above right] {$(x_{i+1}, y_{j+1})$}
    (0, 6) node[above left] {$(x_i, y_{j+1})$}
    (3, 3) node[] {$\greysquare_{ij}$}
    (0, 3) node[left] {\rotatebox[origin=c]{+90}{$\tilde{u}=u^-_{\Gamma_{ij}^{w}}$}}
    (0, 3) node[right] {\rotatebox[origin=c]{+90}{$\tilde{q_1}={q_1}^{+}_{\Gamma_{ij}^{w}}, \, \tilde{q_2}={q_2}^{+}_{\Gamma_{ij}^{w}}$}}
    (3, 0) node[below] {$\tilde{u}=u^-_{\Gamma_{ij}^{s}}$}
    (3, 0) node[above] {$\tilde{q_1}={q_1}^{+}_{\Gamma_{ij}^{s}}, \, \tilde{q_2}={q_2}^{+}_{\Gamma_{ij}^{s}}$}
    (6, 3) node[left] {\rotatebox[origin=c]{+90}{$\tilde{u}=u^-_{\Gamma_{ij}^{e}}$}}
    (6, 3) node[right] {\rotatebox[origin=c]{+90}{$\tilde{q_1}={q_1}^{+}_{\Gamma_{ij}^{e}}, \, \tilde{q_2}={q_2}^{+}_{\Gamma_{ij}^{e}}$}}
    (3, 6) node[below] {$\tilde{u}=u^-_{\Gamma_{ij}^{n}}$}
    (3, 6) node[above] {$\tilde{q_1}={q_1}^{+}_{\Gamma_{ij}^{n}}, \, \tilde{q_2}={q_2}^{+}_{\Gamma_{ij}^{n}}$}
  ;
\end{tikzpicture}
\end{center}
\caption{Definition of alternating numerical flux at the boundaries of interior volumes.}
\label{fig:alternatingUleftQright}
\end{figure}

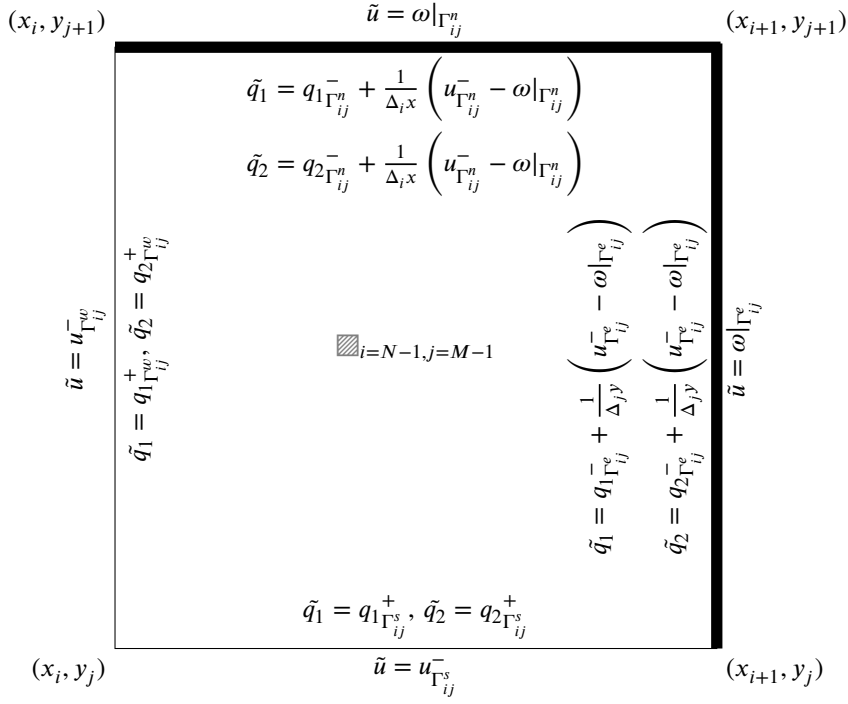
\begin{figure}[!htb]
\begin{center}
\begin{tikzpicture}
  \draw[line width=1.5mm, black] (8,0) -- (8,8.05);
  \draw[line width=1.5mm, black] (0,8) -- (8.05,8);
  \draw (0, 0) rectangle (8, 8);
  \path
    (0, 0) node[below left] {$(x_i, y_j)$}
    (8, 0) node[below right] {$(x_{i+1}, y_j)$}
    (8, 8) node[above right] {$(x_{i+1}, y_{j+1})$}
    (0, 8) node[above left] {$(x_i, y_{j+1})$}
    (4, 4) node[] {$\greysquare_{i=N-1,j=M-1}$}
    (0, 4) node[left] {\rotatebox[origin=c]{+90}{$\tilde{u}=u^-_{\Gamma_{ij}^{w}}$}}
    (0, 4) node[right] {\rotatebox[origin=c]{+90}{$\tilde{q_1}={q_1}^{+}_{\Gamma_{ij}^{w}}, \, \tilde{q_2}={q_2}^{+}_{\Gamma_{ij}^{w}}$}}
    (4, 0) node[below] {$\tilde{u}=u^-_{\Gamma_{ij}^{s}}$}
    (4, 0) node[above] {$\tilde{q_1}={q_1}^{+}_{\Gamma_{ij}^{s}}, \, \tilde{q_2}={q_2}^{+}_{\Gamma_{ij}^{s}}$}
    (8, 4) node[right] {\rotatebox[origin=c]{+90}{$\tilde{u}=\omega\lvert_{\Gamma_{ij}^{e}}$}}
    (7, 3.5) node[left] {\rotatebox[origin=c]{+90}{$\tilde{q_1}={q_1}^{-}_{\Gamma_{ij}^{e}}+\frac{1}{\Delta_j y} \left( u^-_{\Gamma_{ij}^{e}}-\omega\lvert_{\Gamma_{ij}^{e}}\right)$}}
    (8, 3.5) node[left] {\rotatebox[origin=c]{+90}{$\tilde{q_2}={q_2}^{-}_{\Gamma_{ij}^{e}}+\frac{1}{\Delta_j y} \left( u^-_{\Gamma_{ij}^{e}}-\omega\lvert_{\Gamma_{ij}^{e}}\right)$}}    
    (4, 8) node[above] {$\tilde{u}=\omega\lvert_{\Gamma_{ij}^{n}}$}
    (4, 8) node[below] {$\tilde{q_1}={q_1}^{-}_{\Gamma_{ij}^{n}}+\frac{1}{\Delta_i x} \left( u^-_{\Gamma_{ij}^{n}}-\omega\lvert_{\Gamma_{ij}^{n}}\right)$}    
    (4, 7) node[below] {$\tilde{q_2}={q_2}^{-}_{\Gamma_{ij}^{n}}+\frac{1}{\Delta_i x} \left( u^-_{\Gamma_{ij}^{n}}-\omega\lvert_{\Gamma_{ij}^{n}}\right)$}    
  ;
\end{tikzpicture}
\end{center}
\caption{Definition of alternating numerical flux at the boundaries of the $\greysquare_{N-1,M-1}$ volume.}
\label{fig:alternatingUleftQrightRightUpCorner}
\end{figure}

Finally, on top of that alternating inversion, penalty terms of the form $\frac{1}{\Delta}(u^--\omega)$ are added in the numerical fluxes for $\tilde{q_1}$ and $\tilde{q_2}$. These penalty expressions are considered to ``enhance the stability and guarantee the optimal accuracy of the scheme'', see \cite{WANG2018164}. Lastly, as an example, the alternating numerical fluxes for the right-up cornered volume $\greysquare_{N-1,M-1}$ are the following (also sketched in Figure \ref{fig:alternatingUleftQrightRightUpCorner}),
\begin{itemize}
    \item On the east boundary of the volume $\greysquare_{i=N-1,j=M-1}$, $$\tilde{u}=\omega\lvert_{\Gamma_{ij}^{e}}, \tilde{q_1}={q_1}^{-}_{\Gamma_{ij}^{e}}+\frac{1}{\Delta_j y} \left( u^-_{\Gamma_{ij}^{e}}-\omega\lvert_{\Gamma_{ij}^{e}}\right), \tilde{q_2}={q_2}^{-}_{\Gamma_{ij}^{e}}+\frac{1}{\Delta_j y} \left( u^-_{\Gamma_{ij}^{e}}-\omega\lvert_{\Gamma_{ij}^{e}}\right).$$
    \item On the west boundary of the volume $\greysquare_{i=N-1,j=M-1}$,  $\tilde{u}=u^-_{\Gamma_{ij}^{w}}$, $\tilde{q_1}={q_1}^{+}_{\Gamma_{ij}^{w}}$,  $\tilde{q_2}={q_2}^{+}_{\Gamma_{ij}^{w}}$.
    \item On the north boundary of the volume $\greysquare_{i=N-1,j=M-1}$, 
    $$\tilde{u}=\omega\lvert_{\Gamma_{ij}^{n}}, \tilde{q_1}={q_1}^{-}_{\Gamma_{ij}^{n}}+\frac{1}{\Delta_i x} \left( u^-_{\Gamma_{ij}^{n}}-\omega\lvert_{\Gamma_{ij}^{n}}\right), \tilde{q_2}={q_2}^{-}_{\Gamma_{ij}^{n}}+\frac{1}{\Delta_i x} \left( u^-_{\Gamma_{ij}^{n}}-\omega\lvert_{\Gamma_{ij}^{n}}\right).$$
    \item On the south boundary of the volume $\greysquare_{i=N-1,j=M-1}$, $\tilde{u}=u^-_{\Gamma_{ij}^{s}}$, $\tilde{q_1}={q_1}^{+}_{\Gamma_{ij}^{s}}$, $\tilde{q_2}={q_2}^{+}_{\Gamma_{ij}^{s}}$.
\end{itemize}

So far the semidiscrete LDG scheme has been defined. By summing up the variational formulations \eqref{eq:variational1}-\eqref{eq:variational3} over all the volumes $\greysquare_{ij}$ ($i=0,\ldots,N-1$, $j=0,\ldots,M-1$) we get the following semidiscrete LDG in the global form: 
\begin{align}
    (u_t,v)_{\hat{\Omega}} &= \mathcal{C}(u,v)+\mathcal{H}(u,v)+\mathcal{D}(q_1,q_2,v),\label{eq:globalVariational1}\\
    (q_1,r_1)_{\hat{\Omega}} &= \mathcal{K}^1(u,r_1),\label{globalVariational2}\\
    (q_2,r_2)_{\hat{\Omega}} &= \mathcal{K}^2(u,r_2),\label{globalVariational3}    
\end{align}
where $$(v,w)_{\hat{\Omega}} = \sum_{\greysquare_{ij}} (v,w)_{\greysquare_{ij}}$$ is the inner product in $L^2(\hat{\Omega})$. Besides $$\mathcal{C}(u,v) = \sum_{\greysquare_{ij}} \mathcal{C}_{\greysquare_{ij}}(u,v),$$ and similarly for $\mathcal{H}$, $\mathcal{D}$, $\mathcal{K}^1$ and $\mathcal{K}^2$.

Finally, we express the numerical solution as $$\hat{u}(x,y)=\sum_{k_1,k_2=0}^k u_{i,j}^{k_1,k_2} \Phi_{ij}^{k_1,k_2}(x,y), \qquad (x,y) \in \greysquare_{ij},$$
and we should solve for the coefficients $u_{i,j}^{k_1,k_2}$. In this work, we take the orthogonal nodal basis defined by the tensor product of the one-dimensional Lagrange interpolation polynomial basis over the $k$ Gauss-Legendre quadrature nodes in the intervals $[x_i,x_{i+1}]$ and $[y_j,y_{j+1}]$, respectively. 
As an example, we show under this setting how to compute $\mathcal{H}_{\greysquare_{ij}}(u,v) = ( h(u), v )_{\greysquare_{ij}}$. Let $\Phi^{k_1,k_2}(\xi_x,\xi_y) = \phi^{k_1}(\xi_x) \phi^{k_2}(\xi_y)$ denote the $(k_1,k_2)$-Lagrange polynomial basis on the $(k_1,k_2)$ Gauss-Legendre quadrature node in the canonical volume $[-1,1]\times[-1,1]$. Let $w_{k_1,k_2}=w_{k_1} w_{k_2}$ denote the weight associated to such quadrature node. Such basis functions are mapped to the volume $\greysquare_{ij}$ with the bijections $T_i^x(\xi_x) = \frac{x_{i+1}+x_i}{2} + \frac{\Delta_i x}{2}\xi_x$, i.e $$\Phi_{ij}^{k_1,k_2}(x,y) =  \phi^{k_1}\circ  \l T_i^x\r^{-1} (x) \,\cdot \, \phi^{k_2}\circ  \l T_j^y\r^{-1} (y). $$ Therefore:
\begin{align*}
 \mathcal{H}_{\greysquare_{ij}}(u,\Phi_{ij}^{k_1,k_2}(x,y)) &= \int_{\greysquare_{ij}} h(u)\Phi_{ij}^{k_1,k_2}(x,y) \,dxdy \\
 & \hspace{-0.2cm} = \dfrac{\Delta_i x}{2}\dfrac{\Delta_j y}{2}\int_{[-1,1]\times[-1,1]} h \l u \l T_i^x(\xi_x),T_j^y(\xi_y) \r \r \, \phi^{k_1}(\xi_x)\phi^{k_2}(\xi_y)\,d\xi_x d\xi_y \\
&\hspace{-0.2cm}\approx \dfrac{\Delta_i x}{2}\dfrac{\Delta_j y}{2} \sum_{k_1,k_2=0}^{k} w_{k_1,k_2}h \l  u \l T_i^x(\xi_{x_{k_1}}),T_j^y(\xi_{y_{k_2}}) \r \r \, \phi^{k_1}(\xi_{x_{k_1}})\phi^{k_2}(\xi_{y_{k_2}})  \\
& \hspace{-0.2cm}= \dfrac{\Delta_i x}{2}\dfrac{\Delta_j y}{2} w_{k_1} w_{k_2} h \l u_{ij}^{k_1,k_2} \r.
\end{align*}
The computations of \eqref{convective-term-1}, \eqref{diffusive-term-1}-\eqref{diffusive-term-3} are left to the reader.

In the next section, the spatial discretization is coupled with the IMEX time marching scheme.

\subsection{IMEX time discretization}\label{subsec:IMEX}

Before the end of this section, a fully-discrete LDG scheme will be presented. Here we follow \cite{Cockburn-Shu-SIAM-JNA-1998, doi:10.1137/140956750,Wang-Shu-Zhang-2016,Haijin-wang-Shu-M2AN-2016}. In order to introduce IMEX schemes we consider the next ODE system:
\begin{align}
    u_t &= \xi(u,t) + \psi(u,t),\label{eq:ODE}\\
    u(t_0) &= u^0,\label{eq:ODEiniCond}
\end{align}
where $\xi(u,t)$ comes from the spatial discretization of convection and source terms, and $\psi(x,t)$ arises from the discretization of the diffusive terms.

Let $\{t^n=n\tau\}_{n=0}^{L}$ be the uniform partition of the time interval $[0,T]$ with time step $\tau$. Non-uniform time partitions could be adopted, although in this work, for simplicity, we consider a constant time step $\tau$. Given $u^n$, the numerical solution of \eqref{eq:ODE}-\eqref{eq:ODEiniCond} at time $t^{n}$, a general 
IMEX scheme builds the solution at time $t^{n+1}$, $u^{n+1}$, by the following means 
\begin{align}
    u^{n,0} =& u^n, \nonumber \\
    u^{n,i} =& u^n + \tau \sum_{j=0}^{i-1} \tilde{a}_{ij} \xi^{n,j} + \tau \sum_{j=0}^{i} a_{ij} \psi^{n,j}, \quad 1\leq i \leq s, \label{eq:IMEX_eq_general_ni}\\
    u^{n+1} =& u^n + \tau \sum_{i=0}^{s} \tilde{b}_{i} \xi^{n,i} + \tau \sum_{i=0}^{s} b_{i} \psi^{n,i}, \nonumber
\end{align}
where $u^{n,i}$ are the intermediate IMEX stages for $i = 1, \dots, s$ and $t^{n,i}=t^n+c_i \tau$. Additionally, $\forall \mathbf{x}\in \hat{\Omega}$, 
\begin{align}
        &\xi^{n,j} = \xi(u^{n,j},  \mathbf{x}, t^{n,j})= - \div{\mathbf{F}(u^{n,j})} +h(u^{n,j},\mathbf{x},t^{n,j}), \label{eq:general-explicit-part-1d} \\ 
        &\psi^{n,j} = \psi(u^{n,j}, \mathbf{x},t^{n,j}) = \div(\mbG(\nabla u^{n,j})).
    \end{align}
Besides, the matrices $\tilde{A} = (\tilde{a}_{ij})$ and $A=(a_{ij})$, $0\leq i,j \leq s$, are $(s+1)\times(s+1)$ matrices such that the resulting scheme is explicit in $\xi$ and implicit in $\psi$. The coefficients $c = (0,c_1,\ldots,c_{s})^{\mathrm{T}}$ are given by the relation $c_i = \sum_{j=0}^{i-1} \tilde{a}_{ij} = \sum_{j=0}^{i} a_{ij}$.
Additionally, $\tilde{b} = (\tilde{b}_0,\tilde{b}_1,\ldots,\tilde{b}_{s})$ and $b = (b_0,b_1,\ldots,b_{s})$ are considered. All in all, the general 
IMEX time marching scheme \eqref{eq:IMEX_eq_general_ni} can be expressed as the following Butcher tableaus
\begin{center}
\begin{tabular}{c|c|c}
$c$  & $\tilde{A}$ & $A$\\
\hline
&  $\tilde{b}$ & $b$ 
\end{tabular}.
\end{center}

In the subsequent Section \ref{sec:bt_gt}, for the sake of simplicity, we present the boundary treatment strategy with the succeeding particular third-order IMEX RK method, whose tableaus are given by (see \cite{CALVO2001535,Haijin-Shu-SIAM-JNA-2016})
\begin{center}
\begin{tabular}{c|cccc|cccc}
$0$ & $0$ & $0$ & $0$ & $0$ & $0$ & $0$ & $0$ & $0$ \\
$\gamma$ & $\gamma$ & $0$ & $0$ & $0$ & $0$ & $\gamma$ & $0$ & $0$ \\
$\frac{1+\gamma}{2}$ & $\frac{1+\gamma}{2}-\alpha_1$ & $\alpha_1$ & $0$ & $0$ & $0$ & $\frac{1-\gamma}{2}$ & $\gamma$ & $0$  \\
$1$ & $0$ & $1-\alpha_2$ & $\alpha_2$ & $0$ & $0$ & $\beta_1$ & $\beta_2$ & $\gamma$  \\
\hline
& $0$ & $\beta_1$ & $\beta_2$ & $\gamma$ & $0$ & $\beta_1$ & $\beta_2$ & $\gamma$  \\
\end{tabular},
\end{center}
where $\gamma$ is the middle root of $6x^3-18x^2+9x-1=0$, $\gamma= \frac{1767732205903}{4055673282236}$, $\beta_1 = -\frac32\gamma^2+4\gamma-\frac14$, $\beta_2=\frac32\gamma^2-5\gamma+\frac54$, $\alpha_1=-0.35$ and $\alpha_2 = \dfrac{\frac13 - 2 \gamma^2 - 2\beta_2 \alpha_1 \gamma}{\gamma (1-\gamma)}$. Note also that the intermediate time levels are defined as $t^{n,1}=t+\gamma\tau$, $t^{n,2}=t+\frac{1+\gamma}{2}\tau$, $t^{n,3} = t^n+\tau$.
Thus, for any function $(v,r_1,r_2) \in \hat{V}\times \hat{V} \times \hat{V}$, given  $(u^{n},q_1^{n},q_2^{n})$,  the numerical solution at the next time level is computed through three intermediate numerical solutions $(u^{n,i},q_1^{n,i},q_2^{n,i})$, $i=1,2,3$, in the following way 
\begin{align}
    (u^{n,1},v)_{\hat{\Omega}} &= (u^{n},v)_{\hat{\Omega}} + \gamma \tau (\mathcal{C}+\mathcal{H})(u^n,v) + \gamma \tau \mathcal{D}(q_1^{n,1},q_2^{n,1},v), \label{eq:imex1}\\
    (u^{n,2},v)_{\hat{\Omega}} &= (u^{n},v)_{\hat{\Omega}} + \l \frac{1+\gamma}{2} - \alpha_1 \r \tau (\mathcal{C}+\mathcal{H})(u^n,v) + \alpha_1 \tau (\mathcal{C}+\mathcal{H})(u^{n,1},v) \nonumber \\ 
    & \hspace{0.5cm}+ \frac{1-\gamma}{2} \tau \mathcal{D}(q_1^{n,1},q_2^{n,1},v) + \gamma \tau \mathcal{D}(q_1^{n,2},q_2^{n,2},v), \label{eq:imex2} \\
    (u^{n,3},v)_{\hat{\Omega}} &= (u^{n},v)_{\hat{\Omega}} + (1-\alpha_2)\tau (\mathcal{C}+\mathcal{H})(u^{n,1},v) + \alpha_2 \tau (\mathcal{C}+\mathcal{H})(u^{n,2},v) \nonumber \\
    & \hspace{0.5cm}+ \beta_1 \tau \mathcal{D}(q_1^{n,1},q_2^{n,1},v) + \beta_2 \tau \mathcal{D}(q_1^{n,2},q_2^{n,2},v) + \gamma \tau \mathcal{D}(q_1^{n,3},q_2^{n,3},v), \label{eq:imex3} \\
    (u^{n+1},v)_{\hat{\Omega}} &= (u^{n},v)_{\hat{\Omega}} + \beta_1\tau (\mathcal{C}+\mathcal{H})(u^{n,1},v) + \beta_2 \tau (\mathcal{C}+\mathcal{H})(u^{n,2},v) + \gamma \tau (\mathcal{C}+\mathcal{H})(u^{n,3},v) \nonumber \\
    & \hspace{0.5cm}+ \beta_1 \tau \mathcal{D}(q_1^{n,1},q_2^{n,1},v) + \beta_2 \tau \mathcal{D}(q_1^{n,2},q_2^{n,2},v) + \gamma \tau \mathcal{D}(q_1^{n,3},q_2^{n,3},v), \label{eq:imexFinal} \\    
    (q_1^{n,i},r)_{\hat{\Omega}} &= \mathcal{K}^1(u^{n,i},r), \quad (q_2^{n,i},r)_{\hat{\Omega}} = \mathcal{K}^2(u^{n,i},r), \qquad i=1,2,3. \label{eq:imexq12_123}
\end{align}

Concerning the boundary conditions, we distinguish two cases. On the one hand, at times $t^n$, one can directly impose the boundary condition $\omega$, i.e. $u^n = \omega(t^n)$. On the other hand, it is well known in the literature that, when an IMEX method is used together with the method of lines for the full discretization of problem \eqref{eq:system1}-\eqref{eq:system2}, the time-dependent boundary condition $\omega$ should not be directly evaluated at the intermediate stage time levels $t^{n,i}$, $1\leq i \leq s$, i.e. $u^{n,i}=\omega(t^{n,i})$. Such a naive strategy generates numerical boundary layers \cite{LUBICH1995241}. The consequence is a serious degradation in the accuracy of the scheme, i.e. the order of convergence is smaller than the conventional order of the considered time integrator. In the next section, we design a strategy to compute appropriate boundary conditions to be imposed at the intermediate IMEX time stages. Such boundary treatment techniques will allow the method to achieve the desired order accuracy when dealing with general convection-diffusion-reaction problems with Dirichlet boundary conditions. 

\section{Boundary treatment, general technique} \label{sec:bt_gt}

In this section, we describe the general boundary treatment technique for parabolic PDEs, both for linear and nonlinear cases. We present the detailed procedure in the scalar case for the IMEX third-order scheme \eqref{eq:imex1}-\eqref{eq:imexq12_123}, for the sake of simplicity. Nevertheless, the strategy can be applied to any IMEX LDG scheme, as it will be shown in Section \ref{sec:numericalAlgs}.
In the supplementary materials \ref{app:A} and \ref{app:B}, we show the application of the general technique to one-dimensional linear and nonlinear PDEs, respectively, both with source terms. Finally, in the supplementary material \ref{app:C} we extend the procedure to the two-dimensional case. All the computations for each stage are fully detailed so that the interested reader can easily reproduce all the results presented in Section \ref{sec:numericalExperiments} of numerical experiments.

We start by considering a generic nonlinear one-dimensional convection-diffusion-re\-ac\-tion PDE in the conservative form:
\begin{equation}
\label{eq:general-parabolic-1d-conservative}
    u_t + f_x(u) =g_x(u_{x}) + h(u,x,t),
\end{equation}
where $f(u)$ is a physical flux, $g_x(u_{x})$ is the diffusive term and $h(u,x,t)$ is the source term. We consider that $f$, $g$ and $h$ are smooth functions.
We can write equation \eqref{eq:general-parabolic-1d-conservative} as follows, 
\begin{align}
    u_t =& - f_x(u) +g_x(u_{x}) + h(u,x,t).    \label{eq:edp-general-1d}
\end{align}

The boundary treatment technique is based on:
\begin{enumerate}
\item Using the equations of the IMEX internal stages. 
\item Differentiating the equations of the IMEX internal stages.
\item Using the Lax-Wendroff approach (also called the Cauchy-Koval\'evskaya procedure, see, for example, \cite{DAKIN2018228,SEAL2017}) to replace certain time derivatives into space derivatives.
\item Doing the numerical approximation of some high-order spatial derivatives of the numerical solution at the boundaries.
\end{enumerate}

As presented in \eqref{eq:boundCond}, $\omega(t)$ is the Dirichlet boundary condition to be applied at time $t$. We remark that the treatment is equivalent in any boundary of the spatial domain. To proceed, we introduce the following result.

\begin{proposition}\label{approximation_tnl}
    Let $H: \Omega\times [0,T] \to \Omega $ be a $C^2$ function, then
    \begin{equation*}
        H(z(x,t+\tau),t+\tau) = H(z(x,t),t) + \mathcal{O}(\tau) \equiv \bar{H}(t, t+\tau).
    \end{equation*}
    for $z = u, u_x$.
\end{proposition}
\textbf{Proof:} We just need to consider the Taylor expansion
\begin{equation*}
 H(z(t+\tau),t+\tau) = H(z(t),t) + \left[ \partial_t H\l z(t),t \r + \frac{\partial H}{\partial z} \l z(t),t \r\frac{\partial z}{\partial t}\l z(t),t \r\right] \tau + \mathcal{O}  \l \tau ^2 \r,  
\end{equation*}
which stands for any function $z$ of $u$ and $u_x$. 

\begin{remark}
    The previous proposition can be extended to an arbitrary order in $\tau$ by adding more terms of the Taylor expansion. This is needed for higher-order IMEX LDG schemes, as we will see in the following subsection. 
\end{remark}

\subsection{First IMEX stage}\label{subsec:IMEX_gen_stage1}
The first IMEX stage \eqref{eq:imex1} then reads:
\begin{equation}
    \frac{u^{n,1} - u^n}{\gamma \tau}= 
    \left[-f_x(u^n) +h(u^n,x,t^n) \right]
    %
    + 
    \left[ g_x(u_{x}^{n,1}) \right].
    \label{eq:IMEX_1-general-1d}
\end{equation}
Now, using \eqref{eq:general-explicit-part-1d}, it can be written as 
\begin{equation*}
    \frac{u^{n,1} - u^n}{\gamma \tau}= 
    \xi^{n,0}
    %
    + 
    \left[ g_x(u_{x}^{n,1}) \right].
\end{equation*}
In the latter, in order to lighten the notation, we are dropping the dependences $(u,x,t)$ of $\xi^{n,l}$.
Then, we consider that equation \eqref{eq:edp-general-1d} is also satisfied in $t^{n,1}$, so
\begin{align*}
     g_x(u_{x}^{n,1}) = &u^{n,1}_t - \xi^{n,1}.
\end{align*}
Plugging this result into \eqref{eq:IMEX_1-general-1d} and solving for $u^{n,1}$, we get
\begin{align}
    u^{n,1} =& u^n + \gamma \tau \l \xi^{n,0}  + u^{n,1}_t -\xi^{n,1}\r,  \label{eq:IMEX_1_1-general-1d}
\end{align}
where $u_t^{n,1}=\partial_t \omega(t^{n,1})$ is the time derivative of the boundary condition evaluated at time $t^{n,1}$. Now, it is important to notice that $\xi^{n,1}$ has dependence on $u_x^{n,1}$ and $u^{n,1}$. On the one hand, $u_x^{n,1}$ is unknown. To avoid the dependence on this derivative, we take derivative concerning $x$ on both sides of \eqref{eq:IMEX_1-general-1d},
    \begin{equation*}
       u_x^{n,1} =  \gamma\tau \l \frac{ u_x^{n}}{\gamma\tau} +\xi^{n,0}_x + g_{xx}(u_{x}^{n,1})\r.  
    \end{equation*}
    And we replace this derivative on $\xi^{n,1}$. We can see that in this substitution, $u^{n,1}_x$ is still involved due to the factor $g_{xx}(u^{n,1}_x)$, but it is now multiplied by $\tau$. Then, this substitution, using Proposition \ref{approximation_tnl} (with $H = g_{xx}$) leads to
    \begin{equation}
    u_x^{n,1} =  \gamma\tau \l \frac{ u_x^{n}}{\gamma\tau} +\xi^{n,0}_x + g_{xx}(u_{x}^{n,1})\r  =  \gamma\tau \l \frac{ u_x^{n}}{\gamma\tau} +\xi^{n,0}_x + g_{xx}(u_{x}^{n,0})\r + \mathcal{O}(\tau^2). \label{eq:uxn1-general-1d}
    \end{equation}
    Using this result on \eqref{eq:IMEX_1_1-general-1d} ensures a third-order time approximation for $u^{n,1}$. For higher order schemes, according to Proposition \ref{approximation_tnl}, we need to compute time derivatives of $u_x^{n,0}$. We can take derivatives on \eqref{eq:edp-general-1d} to compute these terms.
    On the other hand, the fact that $\xi^{n,1}$ has a dependence on $u^{n,1}$ forces us to solve in general a nonlinear equation. Once $u_x^{n,1}$ is correctly approximated as described just before, we propose two possible ways to deal with this equation:
    \begin{enumerate}[leftmargin=0.5cm] 
        \item Solve directly the nonlinear equation, either analytically or numerically. The resulting equation (or system of equations) reads
        \begin{equation*}
            u^{n,1}  = \zeta(u^{n,1}),
        \end{equation*}
        which could be solved with a fixed point iterative method, Newton method, or any other nonlinear solver.
        \item Use Proposition \ref{approximation_tnl}: by setting $H$ as the identity function, we can find the suitable approximation of $u^{n,1}$ in terms of $u^{n,0}$ and its derivatives. As $\xi^{n,1}$ is multiplied by $\tau$ we only need to consider $p-1$ time derivatives of $u^{n,0}$ for a $p$-th order IMEX LDG scheme. In general, for a third-order IMEX-LDG scheme the approximation of $u^{n,l}$ involved in $\xi^{n,l}$ is
        \begin{equation}\label{eq:approx_nl_a_n0}
            u^{n,l} = u^{n,0} + u_t^{n,0}c_l\tau + \mathcal{O}(\tau^2). 
        \end{equation}
        We remark that this approximation may only be substituted in the terms involved on $\xi^{n,l}$. 
    \end{enumerate}

\begin{remark}
    In \eqref{eq:IMEX_1_1-general-1d} we are assuming that the time derivatives of the boundary condition can be computed at any intermediate time. If this is not possible, one can consider a Cauchy-Koval\'evskaya procedure to compute the time derivatives with spatial derivatives in $t^{n,0}$. Also, in case we have point-wise data of the boundary condition, an interpolation procedure with enough order may be applied in order to interpolate the values at the time points where the boundary condition must be evaluated.
\end{remark}

\begin{remark} \label{remark:finiteDiff}
    Note that in equation \eqref{eq:uxn1-general-1d}, in general we need to compute a numerical approximation for $u^n_{x}$, $u^n_{xx}$ and $u^n_{xxx}$. For a $p$-th order LDG scheme, to recover the $p$-th order accuracy, we need to consider approximations of that derivatives of at least order $\mathcal{O}(\Delta x^{p-2})$, since in the boundary treatment equations for the internal IMEX stages, these spatial derivatives are all multiplied by $\tau^2$. In such a situation, the time step $\tau$ has to be $\tau=\mathcal{O}(\Delta x)$, to recover the desired accuracy also in time. Therefore, we can use the approximations $u_x(x)=\hat{u}_x(x)$ and $u_{xx}(x)=\hat{u}_{xx}(x)$, with $x\in\Gamma(\Omega)$, since the first and the second derivatives are order $p-1$ and $p-2$, respectively. For the third derivative, since $\hat{u}_{xxx}$ is only $\mathcal{O}(\Delta x^{p-3})$, we can not directly rely on the numerical derivative of the local polynomial solution of the boundary volumes. Therefore one must construct some valid numerical approximation using information from the neighbouring volumes. Below we detail first-order approximation that could be used for third-order schemes. Let $\l\hat{u}_x\r_{i\pm a} := \l\hat{u}_x\r_{i\pm a}(x \pm a \Delta_i x)$. We consider:
    \begin{itemize}
        \item If $x$ is at the left boundary: 
        \begin{align*}
            \hat{u}_{xxx}(x) &\approx \dfrac{\l\hat{u}_x\r_{i+2} - 2 \l\hat{u}_x\r_{i+1} + \l\hat{u}_x\r_{i} }{(\Delta_i x)^2}.
        \end{align*}
        \item If $x$ is at the right boundary: 
        \begin{align*}
            \hat{u}_{xxx}(x) &\approx \dfrac{\l\hat{u}_x\r_{i} - 2 \l\hat{u}_x\r_{i-1} + \l\hat{u}_x\r_{i-2} }{(\Delta_i x)^2}.
        \end{align*}
    \end{itemize}

\end{remark}

\subsection{Second IMEX stage}\label{subsec:IMEX_gen_stage2}
The second IMEX stage \eqref{eq:imex2} reads
\begin{align}
    \frac{u^{n,2}-u^n}{\tau} =& \l \frac{1 + \gamma}{2} - \alpha_1\r \xi^{n,0} + \alpha_1 \xi^{n,1} + \frac{1 - \gamma}{2}g_x(u^{n,1}_{x}) + \gamma g_x(u^{n,2}_{x}).
    \label{IMEX-stage2-general-1d}
\end{align}
As in the previous stage, we impose that the PDE \eqref{eq:edp-general-1d} is satisfied at times $t^{n,1}$ and $t^{n,2}$, then, we substitute the diffusive terms leading to
\begin{align*}
    \frac{u^{n,2}-u^n}{\tau} = \l \frac{1 + \gamma}{2} - \alpha_1\r \xi^{n,0} + \alpha_1 \xi^{n,1} + \frac{1 - \gamma}{2}\left[ u_t^{n,1} - \xi^{n,1} \right] + \gamma \left[u_t^{n,2} -\xi^{n,2} \right] .
\end{align*}
Rewriting the terms we end up with
\begin{align*}
    \frac{u^{n,2}-u^n}{\tau} = \frac{1 - \gamma}{2} u_t^{n,1} + \gamma u_t^{n,2} 
     +\l \frac{1 - \gamma}{2} -\alpha_1 \r \l \xi^{n,0} 
     - \xi^{n,1} \r
     + \gamma \xi^{n,0} 
    - \gamma \xi^{n,2}  ,
\end{align*}
where $ \xi^{n,0} - \xi^{n,1}$ can be substituted using \eqref{eq:IMEX_1_1-general-1d} which is the first IMEX stage equation. Finally
\begin{align}
     \frac{u^{n,2}-u^n}{\tau} =  \alpha_1 u_t^{n,1} + \gamma u_t^{n,2} + \l \frac{1 - \gamma}{2} -\alpha_1 \r \frac{u^{n,1} - u^n}{\gamma \tau}   
      + \gamma \xi^{n,0} 
    - \gamma \xi^{n,2} .  \label{eq:IMEX_2_2-general-1d}
\end{align}
Again, we need to solve this nonlinear system of equations which might have a dependence on $u_x^{n,2}$, which can be computed by taking derivatives on the second IMEX stage \eqref{IMEX-stage2-general-1d}:
\begin{equation*}
    u_x^{n,2} =  u_x^n + \tau \left[\l \frac{1 + \gamma}{2} - \alpha_1\r \xi_x^{n,0} + \alpha_1 \xi_x^{n,1} + \frac{1 - \gamma}{2}g_{xx}(u^{n,1}_{x}) + \gamma g_{xx}(u^{n,2}_{x}) \right]. 
\end{equation*}
According to Proposition \ref{approximation_tnl} for all the terms in $t^{n,1}$ and $t^{n,2}$ and regrouping
\begin{align*}
    u_x^{n,2} &=  u_x^n + \tau \left[ \l \frac{1 + \gamma}{2}\r \xi_x^{n,0}  + \frac{1 + \gamma}{2}g_{xx}(u^{n,0}_{x})  \right] + \mathcal{O}(\tau^2).
\end{align*}
Now notice that the last term is exactly \eqref{eq:uxn1-general-1d}, which is already computed:
\begin{equation} \label{eq:IMEX_2_2-u_x}
    u_x^{n,2} =  u_x^n +\l \frac{1 + \gamma}{2}\r \frac{u_x^{n,1} - u_x^{n,0}}{\gamma}.
\end{equation}
Then we substitute this expression on $\xi^{n,2}$ in equation \eqref{eq:IMEX_2_2-general-1d} leading to a nonlinear problem on $u^{n,2}$ that may be solved the same way as described in the first stage.
The result in \eqref{eq:IMEX_2_2-u_x} is generalized for any intermediate step of any IMEX scheme in the next Theorem \ref{theo:aproximacionetapaintermedia}.
\begin{theorem}\label{theo:aproximacionetapaintermedia}
    Let $u^{n,i}$ be the solution at the intermediate IMEX time $t^{n,i}$ 
 for $i \geq 1$ according to scheme \eqref{eq:IMEX_eq_general_ni}. Then, the approximation of the first derivative, according to Proposition \ref{approximation_tnl} is
    \begin{align*}
        u_x^{n,i} &= u_x^{n,0} + \tau \Bigg[c_{{i}} \frac{u_x^{n,1}-u_x^{n,0}}{c_1 \tau} + \sum_{j = 0}^{i-1} \tilde{a}_{ij}\l \overline{\xi_x}(t^{n,0},t^{n,j}) -\xi_x(t^{n,0})\r  \\
        & \hspace{2cm}+ \sum_{j = 0}^{i} {a}_{ij}\l \overline{\psi_{x}}(t^{n,0},t^{n,j}) - \overline{\psi_{x}}(t^{n,0},t^{n,1}) \r\Bigg].
    \end{align*}
    We recall that $\overline{\xi_x}(t^{n,0},t^{n,j})$ and $\overline{\psi_{x}}(t^{n,0},t^{n,j})$ are the approximations for the time $t^{n,j}$ in terms of $t^{n,0}$ for the explicit and implicit parts, respectively.
\end{theorem}
\begin{proof}
    We take the derivative on \eqref{eq:IMEX_eq_general_ni} which leads us to:
\begin{align*}
         \frac{u_x^{n,i} - u_x^{n,0}}{\tau} =& \sum_{{j=0}}^{i-1} \tilde{a}_{{i}j} \xi_x(t^{n,j}) +  \sum_{j = 0}^{i} {a}_{{i}j} \psi_{x}(t^{n,j}) \\
        =& \sum_{j = 0}^{i-1} \tilde{a}_{ij} \xi_x(t^{n,0})   + \sum_{j = 0}^{i} {a}_{ij}\overline{\psi_{x}}(t^{n,0},t^{n,1}) \\ 
         & + \sum_{j = 0}^{i-1} \tilde{a}_{ij}\l \overline{\xi_x}(t^{n,0},t^{n,j}) -  \xi_x(t^{n,0}) \r  + \sum_{j = 0}^{i} {a}_{ij}\l \overline{\psi_{x}}(t^{n,0},t^{n,j}) - \overline{\psi_{x}}(t^{n,0},t^{n,1}) \r \\
         =& c_{i} \l  \xi_x(t^{n,0})+ \overline{\psi_{x}}(t^{n,0},t^{n,1})\r \\ 
         & + \sum_{j = 0}^{i-1} \tilde{a}_{ij}\l \overline{\xi_x}(t^{n,0},t^{n,j}) -  \xi_x(t^{n,0}) \r  + \sum_{j = 0}^{i} {a}_{ij}\l \overline{\psi_{x}}(t^{n,0},t^{n,j}) - \overline{\psi_{x}}(t^{n,0},t^{n,1}) \r \\
         =& c_{i} \frac{u_x^{n,1}-u_x^{n,0}}{c_1 \tau} \\ 
         & + \sum_{j = 0}^{i-1} \tilde{a}_{ij}\l \overline{\xi_x}(t^{n,0},t^{n,j}) - \xi_x(t^{n,0}) \r  + \sum_{j = 0}^{i} {a}_{ij}\l \overline{\psi_{x}}(t^{n,0},t^{n,j}) - \overline{\psi_{x}}(t^{n,0},t^{n,1}) \r ,
    \end{align*}
where we have taken into account that $$\frac{u_x^{n,1}-u_x^n}{c_1\tau} = \xi_x(t^{n,0}) + \overline{\psi_x}(t^{n,0},t^{n,1}) .$$

We can compute further the last term
\begin{align*}
    \overline{\psi_{x}}(t^{n,0},t^{n,j})- \overline{\psi_{x}}(t^{n,0},t^{n,1})&=\sum_{r=0}^{k-2}\frac{1}{r!}\frac{\partial^r}{\partial t^r}\psi_{x} (t^{n,0})\Big[(t^{n,j}-t^{n,0})^r - (t^{n,1}-t^{n,0})^r\Big] \\
    &=\sum_{r=0}^{k-2}\frac{1}{r!}\frac{\partial^r}{\partial t^r}\psi_{x} (t^{n,0})\Big[(c_j \tau)^r - (c_1\tau)^r\Big] \\
    &=\sum_{r=0}^{k-2}\frac{\tau^r}{r!}(c_j^r - c_1^r)\frac{\partial^r}{\partial t^r}\psi_{x} (t^{n,0}).    
\end{align*}
    
\end{proof}

\begin{remark}
    For a $p$-th order IMEX scheme, it is needed at least a $p-2$-th order approximation of the spatial derivative. Then, for second and third-order IMEX schemes it is only needed a first-order approximation using Proposition \ref{approximation_tnl}. In this case, last terms vanished due to $\overline{H}(t^{n,j}, t^{n,0}) = \overline{H}(t^{n,1}, t^{n,0})$. 
\end{remark}

\subsection{Third IMEX stage}\label{subsec:IMEX_gen_stage3}
The third IMEX stage \eqref{eq:imex3} reads
\begin{equation}
    \frac{u^{n,3} - u^n}{\tau} = \l 1 -\alpha_2 \r  \xi^{n,1}  + \alpha_2 \xi^{n,2} + \beta_1 g_x( u_x^{n,1}) + \beta_2 g_x( u_x^{n,2}) +\gamma g_x( u_x^{n,3}). \label{eq:imexStep3-general-1d}
\end{equation}
We impose the PDE to be valid on $t^{n,i}$ for $i=1,2,3$ and substitute the diffusive terms
\begin{align*}
    \frac{u^{n,3} - u^n}{\tau} =& \l 1 -\alpha_2 \r  \xi^{n,1}  + \alpha_2  \xi^{n,2}  + \beta_1 \l u_t^{n,1} - \xi^{n,1}\r + \beta_2  \l u_t^{n,2} - \xi^{n,2}\r +\gamma  \l u_t^{n,3} - \xi^{n,3}\r \\
    =& \beta_1 u_t^{n,1} + \beta_2 u_t^{n,2} + \beta_3 u_t^{n,3} - \l \alpha_2 + \beta_2 -1 \r \xi^{n,1} - \l \beta_2 - \alpha_2 \r \xi^{n,2} - \gamma \xi^{n,3}.
\end{align*}
Then, using \eqref{eq:IMEX_1_1-general-1d} and \eqref{eq:IMEX_2_2-general-1d} we can substitute $\xi^{n,i}$:
\begin{align*}
      \xi^{n,1} &= \frac{u^{n,1} - u^n}{\gamma \tau} - \xi^{n,0} - u^{n,1}_t, \\
     \xi^{n,2} &= \frac{u^{n,2}-u^n}{\gamma \tau} - \dfrac{\alpha_1}{\gamma } u_t^{n,1}- u_t^{n,2}- \l \frac{1 - \gamma}{2} -\alpha_1 \r \frac{u^{n,1} - u^n}{\gamma^2 \tau}  - \xi^{n,0}.
\end{align*}
Substituting these expressions, regrouping and using that $\beta_1 + \beta_2 + \gamma = 1$:
\begin{align}
    \frac{u^{n,3} - u^n}{\tau} =& \left[ 1-\alpha_2 - \l \beta_2 - \alpha_2 \r \frac{\alpha_1}{\gamma} \right]u_t^{n,1} + \alpha_2u_t^{n,2} + \gamma u_t^{n,3} +\l \alpha_2 + \beta_1 -1 \r \frac{u^{n,1} - u^n}{\gamma \tau} \nonumber \\ 
    &+ \l \beta_2 - \alpha_2 \r \left[  \frac{u^{n,2}-u^n}{\gamma \tau}  - \l \frac{1 - \gamma}{2} -\alpha_1 \r \frac{u^{n,1} - u^n}{\gamma^2 \tau} \right]  \nonumber \\
    &+\gamma \xi^n - \gamma \xi^{n,3} .\label{eq:IMEX_3_general}
\end{align}
Again, we have a dependency of $u_x^{n,3}$ on $\xi^{n,3}$ which can be vanished by taking the derivative in \eqref{eq:imexStep3-general-1d}, taking into account the order on $\tau$ and $\beta_1 + \beta_2 + \gamma = 1$, it results in
\begin{align}\label{eq:approx_x_n3_x_n}
     \frac{u_x^{n,3}-u_x^n}{\tau} = \xi_x^{n,0}  +  g_{xx}(u_x^{n,0}) + \mathcal{O}\l \tau \r = \frac{u_{x}^{n,1}-u_{x}^{n,0}}{\gamma \tau} + \mathcal{O}\l \tau \r,
\end{align}
which is the result we stated in Theorem \ref{theo:aproximacionetapaintermedia}.

Finally, substituting this result on $\xi^{n,3}$ involved on \eqref{eq:IMEX_3_general} leads to a nonlinear problem on $u^{n,3}$ which, again, might be solved as in the first stage.

\section{Numerical algorithms for the boundary treatment}\label{sec:numericalAlgs}

The computations described in the previous section can become quite convoluted and tedious to be done by hand, in particular for the last internal IMEX stages. On top of that, those computations are exhausting for schemes of higher order. 
In this section, we propose a completely algorithmic version of the presented technique, that allows us to fully automate the boundary treatment strategy,
by letting the computer numerically perform all the computations. The proposed general algorithms can be applied to any IMEX scheme, as long as the IMEX tableaus are supplied.

Algorithm \ref{alg:algShu} is a generalization of the procedure described previously for IMEX time integration schemes of arbitrary order. For third-order IMEX schemes, as stated before, in the Taylor expansions in time for passing the space derivatives of $u^{n,i}$ to time $t^{n,0}$, it is enough to truncate the series to its first zero order term. For schemes of higher order, more terms have to be considered. In such a situation, mixed (space-time) derivatives of $u^{n,i}$ have to be computed. The way we propose to deal with these terms is to make use of the PDE itself to convert the mixed derivatives to just spatial derivatives, which can be computed using the LDG solution as done before. Besides, we point out that in this algorithm, the terms $u^{n,i}$ coming from the source term of the PDE are denoted as
$u^{n,i}_h$. To avoid having to solve, eventually, nonlinear equations numerically, Taylor expansion to time $t^{n,0}$ is here considered. Finally, notice that stages 3, 5, and 7 are the suitable Taylor expansions of the terms approximated according to Proposition \ref{approximation_tnl}, and stage 6 is the result proved in Theorem \ref{theo:aproximacionetapaintermedia}. All in all, Algorithm \ref{alg:algShu} is a generalization of the work proposed by Shu et al. which opens the door to naturally implement the boundary treatment strategy detailed in Section \ref{sec:bt_gt}. {Therefore, the procedure is fully numerical, letting the machine deal with the computations, without the hassle of the analytical calculations and the corresponding simplifications}.

\begin{algorithm}
    \caption{Boundary treatment approximating the terms $\xi_x^{n,i}$ with $\xi_x^{n,0}$, and $\psi_x^{n,i}$ with $\psi_x^{n,0}$, $u_t = \xi(u,t) + \psi(u,t)$, one-dimensional problem.}
    \begin{algorithmic}[1]
    \STATE $\xi^{n} = \xi^{n,0} \gets \xi^{n,0}(\hat{u}_x^{n,0},u^{n,0}) = \xi^{n,0}(\hat{u}_x^{n},\omega(t^n)) $\vspace{0.2cm}
    \FOR {$i \gets 1...,s$} \vspace{0.2cm}
        \STATE $ {u_{xx}^{n,i-1}} \gets \displaystyle \sum_{r=0}^{k-2} \frac{\tau^r}{r!} \partial_t^r \hat{u}_{xx}^{n,0} c_{i-1}^r$ (Taylor expansion). For 3rd order ($k=2$) $,{u_{xx}^{n,i-1}} \gets \hat{u}_{xx}^{n,0}$ \vspace{0.2cm}
        \STATE ${\xi_x^{n,i-1}} \gets \xi_x^{n,i-1}(u_{xx}^{n,i-1},{u_x^{n,i-1}})$ \vspace{0.2cm}
        \STATE ${\psi_x^{n,i}} \gets \displaystyle \sum_{r=0}^{k-2} \frac{\tau^r}{r!} \partial_t^r {\psi}^{n,0}_{x} c_i^r$ (Taylor expansion). For 3rd order ($k=2$) $,{\psi_x^{n,i}} \gets \hat{\psi}_x^{n,0}$ \vspace{0.2cm}
        \STATE ${u_x^{n,i}} \gets \hat{u}_x^n + \tau \displaystyle\sum_{j=0}^{i-1} \tilde{a}_{ij} {\xi_x^{n,j}} + \tau \sum_{j=0}^{i} a_{ij} {\psi_x^{n,j}}$ \vspace{0.2cm}
        \STATE $\displaystyle {u_h^{n,i}} \gets \sum_{r=0}^{k-1} \frac{\tau^r c_i^r}{r!}\partial_t^r u^n$ (Taylor expansion). For 3rd order, { \small ${u_h^{n,i}} \gets u^n + c_i \tau u_t^n =  \omega(t^n) + c_i \tau \omega_t(t^n)$ } \vspace{0.2cm}
               \STATE ${\xi^{n,i}} \gets \xi^{n,i}({u_x^{n,i}},{u_h^{n,i}})$ \vspace{0.2cm}
        \STATE $\displaystyle u^{n,i} \gets u^n + \tau \sum_{j=0}^{i-1} \tilde{a}_{ij} {\xi^{n,j}} + \tau \sum_{j=0}^{i} a_{ij} (u_t^{n,j}-{\xi^{n,j}})$  \vspace{0.2cm}
    \ENDFOR
    \end{algorithmic}
     \label{alg:algShu}
\end{algorithm}

At this point, in Algorithm \ref{alg:generic1d} we present an improvement of Algorithm \ref{alg:algShu}. Firstly, we note that it is not necessary to move from time $t^{n,i}$ to time $t^{n,0}$ by Taylor expansion of the $\xi$ terms treated explicitly by IMEX. They can be directly evaluated in their natural IMEX times $t^{n,i}$. The advantages are twofold. On the one hand, errors coming from the truncation in the Taylor series are avoided. On the other hand, the implementation is easier, especially for IMEX schemes of order greater than three, since we circumvent the computation of time-space derivatives for the $\xi$ terms. The second improvement in Algorithm \ref{alg:generic1d} is related to the treatment of the $\psi_x$ terms treated implicitly by IMEX. At the $i$-th IMEX intermediate stage, the terms $\psi_x^{n,j}$ for $j<i$ are directly computed using the corresponding intermediate LDG solution at the same $t^{n,j}$ time. Besides, for the terms $\psi_x^{n,i}$, we move them to the immediately preceding time $t^{n,i-1}$ by Taylor expansion. Finally, note that the terms $u^{n,i}$ coming from the reaction part of the PDE, denoted in the algorithms by $u_h^{n,i}$, are treated in the same way as in Algorithm \ref{alg:algShu}. The reason behind this choice is that the Taylor expansions of these terms only require the computation of analytical time derivatives of $\omega$.

\begin{algorithm}
    \caption{Boundary treatment just approximating the terms $\psi_x^{n,i}$ with $\psi_x^{n,i-1}$, $u_t = \xi(u,t) + \psi(u,t)$, one-dimensional problem.}
    \begin{algorithmic}[1]
    \STATE $\xi^{n} = \xi^{n,0} \gets \xi^{n,0}(\hat{u}_x^{n,0},u^{n,0}) = \xi^{n,0}(\hat{u}_x^{n},\omega(t^n)) $ \vspace{0.2cm}
    \FOR {$i \gets 1...,{s}$} \vspace{0.2cm}
        \STATE $\xi_x^{n,i-1} \gets \xi_x^{n,i-1}(\hat{u}_{xx}^{n,i-1},{u_x^{n,i-1}})$ \vspace{0.2cm}
        \STATE ${\psi_x^{n,i}} \gets \displaystyle \sum_{r=0}^{k-2} \frac{\tau^r}{r!} \partial_t^r \hat{\psi}^{n,i-1}_{x} (c_i-c_{i-1})^r$.  For 3rd order ($k=2$) ${\psi_x^{n,i}} \gets \hat{\psi}_x^{n,i-1}$ \vspace{0.2cm}
        \STATE ${u_x^{n,i}} \gets \hat{u}_x^n + \tau \displaystyle\sum_{j=0}^{i-1} \tilde{a}_{ij} \xi_x^{n,j} + \tau \sum_{j=0}^{i-1} a_{ij} \hat{\psi}_x^{n,j}+a_{ii}{\psi_x^{n,i}}$ \vspace{0.2cm}
        \STATE $\displaystyle {u_h^{n,i}} \gets \sum_{r=0}^{k-1} \frac{\tau^r c_i^r}{r!}\partial_t^r u^n$. For 3rd order, ${u_h^{n,i}} \gets u^n + c_i \tau u_t^n = \omega(t^n) + c_i \tau \omega_t(t^n)$  \vspace{0.2cm}
        \STATE $\xi^{n,i} \gets \xi^{n,i}({u_x^{n,i}},{u_h^{n,i}})$ \vspace{0.2cm}
        \STATE $\displaystyle u^{n,i} \gets u^n + \tau \sum_{j=0}^{i-1} \tilde{a}_{ij} \xi^{n,j} + \tau \sum_{j=0}^{i} a_{ij} (u_t^{n,j}-\xi^{n,j})$ \vspace{0.2cm}
    \ENDFOR
    \end{algorithmic}
    \label{alg:generic1d}
\end{algorithm}

Finally, the extension of Algorithm \ref{alg:generic1d} to the high dimensional case is straightforward. It is shown in Algorithm \ref{alg:generic2d} for the two-dimensional case in structured meshes.

\begin{algorithm}
    \caption{Boundary treatment, $u_t = \xi(u,t) + \psi(u,t)$, two-dimensional problem.}
    \begin{algorithmic}[1]
    \STATE $\xi^{n} = \xi^{n,0} \gets \xi^{n,0}(\hat{u}_x^{n,0},\hat{u}_y^{n,0},u^{n,0}) = \xi^{n,0}(\hat{u}_x^{n},\hat{u}_y^{n},\omega(t^n)) $ \vspace{0.2cm}
    \FOR {$i \gets 1...,{s}$} \vspace{0.2cm}
        \STATE ${\xi_x^{n,i-1}} \gets \xi_x^{n,i-1}(\hat{u}_{xx}^{n,i-1},\hat{u}_{yx}^{n,i-1},u_x^{n,i-1})$, ${\psi_x^{n,i}} \gets \displaystyle \sum_{r=0}^{k-2} \frac{\tau^r}{r!} \partial_t^r \hat{\psi}^{n,i-1}_{x} (c_i-c_{i-1})^r$ \vspace{0.2cm}
        \STATE ${u_x^{n,i}} \gets \hat{u}_x^n + \tau \displaystyle\sum_{j=0}^{i-1} \tilde{a}_{ij} \xi_x^{n,j} + \tau \sum_{j=0}^{i-1} a_{ij} \hat{\psi}_x^{n,j}+a_{ii}\psi_x^{n,i}$ \vspace{0.2cm}
        \STATE ${\xi_y^{n,i-1}} \gets \xi_y^{n,i-1}(\hat{u}_{xy}^{n,i-1},\hat{u}_{yy}^{n,i-1},u_y^{n,i-1})$, ${\psi_y^{n,i}} \gets \displaystyle \sum_{r=0}^{k-2} \frac{\tau^r}{r!} \partial_t^r \hat{\psi}^{n,i-1}_{y} (c_i-c_{i-1})^r$ \vspace{0.2cm}
        \STATE $u_y^{n,i} \gets \hat{u}_y^n + \tau \displaystyle\sum_{j=0}^{i-1} \tilde{a}_{ij} \xi_y^{n,j} + \tau \sum_{j=0}^{i-1} a_{ij} \hat{\psi}_y^{n,j}+a_{ii}{\psi_y^{n,i}}$ \vspace{0.2cm}
        \STATE $\displaystyle u_h^{n,i} \gets \sum_{r=0}^{k-1} \frac{\tau^r c_i^r}{r!}\partial_t^r u^n$ \vspace{0.2cm}
        \STATE ${\xi^{n,i}} \gets \xi^{n,i}({u_x^{n,i}},{u_y^{n,i}},{u_h^{n,i}})$ \vspace{0.2cm}
        \STATE $\displaystyle u^{n,i} \gets u^n + \tau \sum_{j=0}^{i-1} \tilde{a}_{ij} \xi^{n,j} + \tau \sum_{j=0}^{i} a_{ij} (u_t^{n,j}-\xi^{n,j})$ \vspace{0.2cm}
    \ENDFOR
    \end{algorithmic}
    \label{alg:generic2d}
\end{algorithm}


\section{Numerical experiments}\label{sec:numericalExperiments}

In this section, we present four numerical experiments to empirically validate that, using the boundary treatment procedure proposed in this work the convergence order is successfully restored.

In all the numerical tests presented in this section, there is no substantial numerical difference in considering the approaches ``1'' or ``2'' when dealing with the nonlinear equations resulting at the end of the boundary treatment of all IMEX stages. Besides, also as expected, no substantial numerical differences are observed between Algorithms \ref{alg:algShu} and \ref{alg:generic1d}, the second one being easier to implement, as explained before. 

The tables of errors and convergence orders will be computed for the $L^1$ norm, as well as for the more demanding $L^2$ and $L^\infty$ norms.

On the one hand, for the first three numerical experiments, the third order was considered, using the IMEX scheme presented in equations \eqref{eq:imex1}-\eqref{eq:imexq12_123}. On the other hand, the ARK4(3)6L[2]SA–ERK and ESDIRK tableaus of \cite{KENNEDY2003139} are employed for the last fourth-order numerical test.

\subsection{One-dimensional convective heat equation, third-order test} \label{subsec:1dheat}

Here we solve the one-dimensional linear PDE $u_t + \partial_x f(u) = \partial_x g(u_x) + h(u)$, where $f = -C u$, $g = D u_x$, $h = (D-1)u$, $C,D\in\R$. For this problem, the boundary treatment strategy is explicitly detailed in the supplementary material \ref{app:A}.
For the spatial domain, we consider $\Omega=[-1,1]$, while the computing time is taken as $T=5$. The corresponding time-dependent Dirichlet boundary conditions are given by the exact solution $u(x,t)=e^{-t} \sin(x+Ct)$.

In Table \ref{tb-1d-linear-source} we show the errors and orders of convergence for the third-order IMEX LDG scheme, without and with boundary treatment. Firstly, in the case without boundary treatment, boundary conditions are naively imposed, meaning that boundary conditions are directly evaluated at the intermediate time steps of the IMEX RK time marching method. As expected, the third-order convergence is far from being achieved with this conventional strategy, especially for the $L^2$ and $L^\infty$ norms. Secondly, in the case of boundary treatment, the technique explained in Section \ref{sec:numericalAlgs} is considered. In this case, the method recovers the third-order accuracy for the three considered norms.

Figure \ref{fig:BL_O3} shows the single time step ($t=\tau$) and the final time ($t=T$) errors in space in the cases without and with boundary treatment (w/o bt and w bt, respectively). The huge spatial errors in the boundary regions, appearing when no boundary treatment is considered, vanish with the proposed treatment.

\begin{table}[ht]
\caption{Errors and orders of accuracy for PDE $u_t + \partial_x f(u) = \partial_x g(u_x) + h(u)$, where $f = -C u$, $g = D u_x$, $h = (D-1)u$, considering $C=0.1$, $D=2$. $CFL=0.25$. Spatial domain $[-1,1]$, computing time $T=5$. Dirichlet boundary conditions given by the exact solution $u(x,t)=e^{-t} \sin(x+Ct)$.}
\begin{center}
\begin{tabular}{ccccccc}
    \hline
    $N$ & \multicolumn{6}{c}{Without boundary treatment}\\
        \cline{2-7}
        & $L^1$ error & $L^1$ order & $L^2$ error & $L^2$ order & $L^\infty$ error & $L^\infty$ order \\
    \hline
    $5$ & $5.03e-05$ & $-$ & $4.86e-05$ & $-$ & $9.60e-05$ & $-$ \\
    \hline
    $10$ & $9.11e-06$ & $2.47$ & $9.56e-06$ & $2.35$ & $2.31e-05$ & $2.06$ \\
    \hline    
    $20$ & $1.63e-06$ & $2.48$ & $1.92e-06$ & $2.32$ & $5.70e-06$ & $2.02$ \\
    \hline    
    $40$ & $2.85e-07$ & $2.52$ & $3.93e-07$ & $2.29$ & $1.42e-06$ & $2.01$ \\
    \hline    
    $80$ & $4.96e-08$ & $2.52$ & $8.11e-08$ & $2.28$ & $3.57e-07$ & $1.99$ \\
    \hline    
    $160$ & $8.61e-09$ & $2.53$ & $1.69e-08$ & $2.26$ & $8.96e-08$ & $1.99$ \\
    \hline
    $320$ & $1.50e-09$ & $2.52$ & $3.53e-09$ & $2.26$ & $2.25e-08$ & $1.99$ \\
    \hline    
    $640$ & $2.62e-10$ & $2.52$ & $7.41e-10$ & $2.25$ & $5.64e-09$ & $2.00$ \\
    \hline    
\end{tabular}

\begin{tabular}{ccccccc}
    \hline
    $N$ & \multicolumn{6}{c}{With boundary treatment}\\
        \cline{2-7}
        & $L^1$ error & $L^1$ order & $L^2$ error & $L^2$ order & $L^\infty$ error & $L^\infty$ order \\
    \hline
    $5$ & $1.17e-05$ & $-$ & $1.39e-05$ & $-$ & $3.18e-05$ & $-$ \\
    \hline
    $10$ & $1.18e-06$ & $3.31$ & $1.30e-06$ & $3.42$ & $3.72e-06$ & $3.1$ \\
    \hline    
    $20$ & $1.41e-07$ & $3.07$ & $1.42e-07$ & $3.19$ & $4.55e-07$ & $3.03$ \\
    \hline    
    $40$ & $1.82e-08$ & $2.95$ & $1.70e-08$ & $3.06$ & $5.67e-08$ & $3.00$ \\
    \hline    
    $80$ & $2.37e-09$ & $2.94$ & $2.13e-09$ & $3.00$ & $7.15e-09$ & $2.99$ \\
    \hline    
    $160$ & $3.08e-10$ & $2.94$ & $2.69e-10$ & $2.99$ & $9.03e-10$ & $2.99$ \\
    \hline
    $320$ & $3.98e-11$ & $2.95$ & $3.40e-11$ & $2.98$ & $1.14e-10$ & $2.99$ \\
    \hline    
    $640$ & $5.08e-12$ & $2.97$ & $4.27e-12$ & $2.99$ & $1.44e-11$ & $2.98$ \\
    \hline    
\end{tabular}
\end{center}
\label{tb-1d-linear-source} 
\end{table}

\begin{figure}[!ht]
\centering
\subfigure{\includegraphics[scale=0.41]{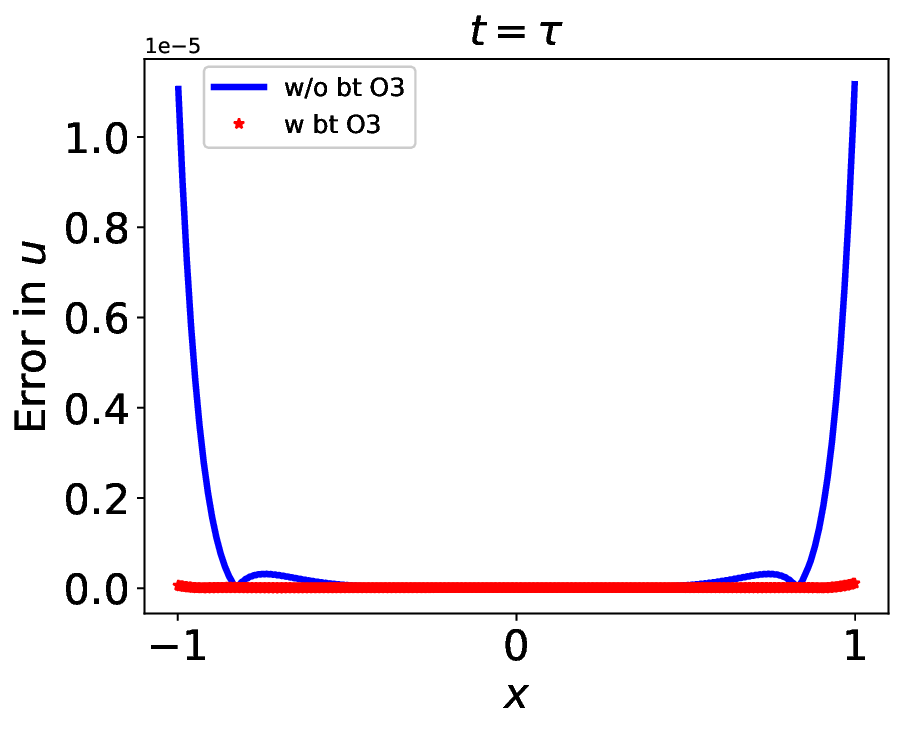}}
\subfigure{\includegraphics[scale=0.41]{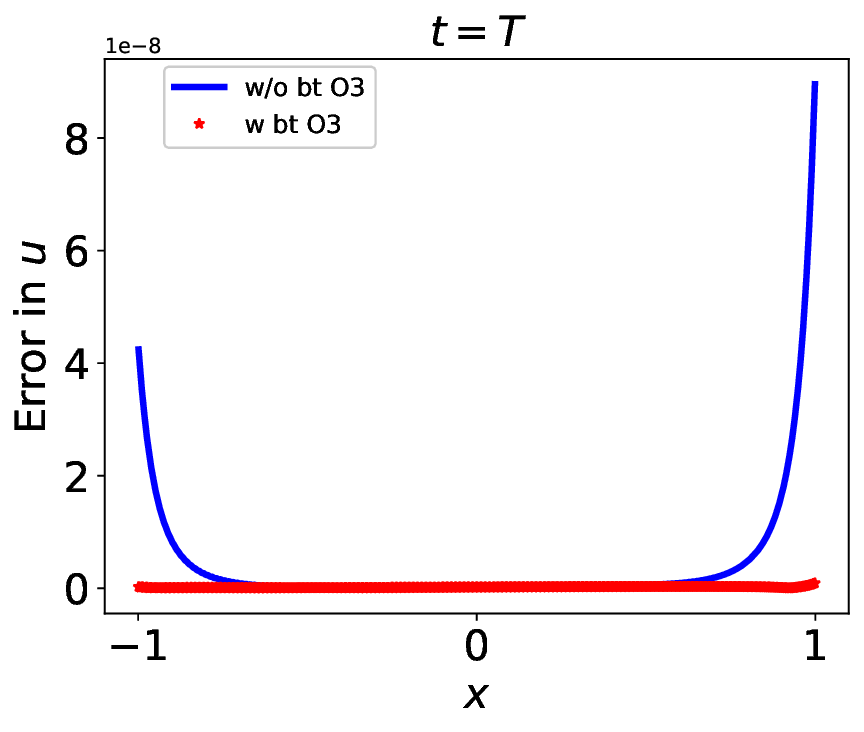}}
\caption{$|\hat{u}-u|$ for $N=160$ with the same data of Table \ref{tb-1d-linear-source}.}
\label{fig:BL_O3}
\end{figure}

\subsection{One-dimensional diffusive Burgers equation, third-order test} \label{subsec:1dburgers}
Now we solve the one-dimensional nonlinear PDE $u_t + \partial_x f(u) = \partial_x g(u_x) + h(u,x,t)$, where $f = \frac{u^2}{2}$, $g = D u_x$, $h = p(x,t)u$, considering $D=2$ and $p(x,t)=D - 1 + e^{-t} \cos x$. The spatial domain is taken as $\Omega=[-1,1]$ and $T=5$ is the computing time. The corresponding Dirichlet boundary conditions are given by the exact solution $u(x,t)=e^{-t} \sin x$. The boundary treatment for this case is fully presented in the supplementary material \ref{app:B}.

\begin{table}[ht]
\caption{Errors and orders of accuracy for PDE $u_t + \partial_x f(u) = \partial_x g(u_x) + h(u,x,t)$, where $f = \frac{u^2}{2}$, $g = D u_x$, $h = p(x,t)u$, considering $D=2$ and $p(x,t)=D - 1 + e^{-t} \cos x$. $CFL=0.4$. Spatial domain $[-1,1]$, computing time $T=5$. Dirichlet boundary conditions given by the exact solution $u(x,t)=e^{-t} \sin x$.}
\begin{center}
\begin{tabular}{ccccccc}
    \hline
    $N$ & \multicolumn{6}{c}{Without boundary treatment}\\
        \cline{2-7}
        & $L^1$ error & $L^1$ order & $L^2$ error & $L^2$ order & $L^\infty$ error & $L^\infty$ order \\
    \hline
    $5$ & $3.10e-06$ & $-$ & $2.30e-06$ & $-$ & $2.95e-06$ & $-$ \\
    \hline
    $10$ & $3.75e-07$ & $3.05$ & $2.90e-07$ & $2.99$ & $5.26e-07$ & $2.49$ \\
    \hline    
    $20$ & $5.19e-08$ & $2.85$ & $4.35e-08$ & $2.74$ & $1.17e-07$ & $2.17$ \\
    \hline    
    $40$ & $7.43e-09$ & $2.80$ & $7.55e-09$ & $2.53$ & $2.88e-08$ & $2.02$ \\
    \hline    
    $80$ & $1.10e-09$ & $2.76$ & $1.45e-09$ & $2.38$ & $7.25e-09$ & $1.99$ \\
    \hline    
    $160$ & $1.66e-10$ & $2.73$ & $2.94e-10$ & $2.30$ & $1.84e-09$ & $1.98$ \\
    \hline
    $320$ & $2.58e-11$ & $2.69$ & $6.09e-11$ & $2.27$ & $4.66e-10$ & $1.98$ \\
    \hline    
    $640$ & $4.10e-12$ & $2.65$ & $1.27e-11$ & $2.26$ & $1.18e-10$ & $1.98$ \\
    \hline    
\end{tabular}

\begin{tabular}{ccccccc}
    \hline
    $N$ & \multicolumn{6}{c}{With boundary treatment}\\
        \cline{2-7}
        & $L^1$ error & $L^1$ order & $L^2$ error & $L^2$ order & $L^\infty$ error & $L^\infty$ order \\
    \hline
    $5$ & $3.03e-06$ & $-$ & $2.19e-06$ & $-$ & $2.19e-06$ & $-$ \\
    \hline
    $10$ & $3.47e-07$ & $3.13$ & $2.49e-07$ & $3.14$ & $2.24e-07$ & $3.29$ \\
    \hline    
    $20$ & $4.23e-08$ & $3.04$ & $3.05e-08$ & $3.03$ & $2.81e-08$ & $2.99$ \\
    \hline    
    $40$ & $5.26e-09$ & $3.01$ & $3.79e-09$ & $3.01$ & $3.51e-09$ & $3.00$ \\
    \hline    
    $80$ & $6.56e-10$ & $3.00$ & $4.73e-10$ & $3.00$ & $4.39e-10$ & $3.00$ \\
    \hline    
    $160$ & $8.20e-11$ & $3.00$ & $5.92e-11$ & $3.00$ & $5.48e-11$ & $3.00$ \\
    \hline
    $320$ & $1.03e-11$ & $2.99$ & $7.39e-12$ & $3.00$ & $6.86e-12$ & $3.00$ \\
    \hline    
    $640$ & $1.28e-12$ & $3.01$ & $9.24e-13$ & $3.00$ & $8.57e-13$ & $3.00$ \\
    \hline    
\end{tabular}

\end{center}
\label{tb-Burgers-withsource} 
\end{table}

In Table \ref{tb-Burgers-withsource} errors and orders of convergence for the third-order IMEX LDG scheme are shown. Once more, the solver without a proper boundary treatment suffers from the order reduction phenomenon. However, the designed boundary treatment strategy can successfully restore third-order accuracy, even in this nonlinear problem with huge temporal variations at the boundaries.

\subsection{Two-dimensional convective heat equation, third-order test\label{sec:exp2d}}
Here the numerical solution of the following two-dimensional linear PDE is carried out
\begin{equation*}
    u_t + \partial_x f_1(u) + \partial_y f_2(u) = \partial_x g_1(u_x, u_y) + \partial_y g_2(u_x, u_y) + h(u),
\end{equation*}
with $f_1 = -Cu$, $f_2 = -Cu$, $g_1 = D u_x$, $g_2 = Du_y$ and $h = (2D-1)u$. $C\in\R$ represents the convection coefficients while $D\in\R$ the diffusion parameters. For the spatial domain we consider $\Omega = [-1,1]\times[-1,1]$, while the computing time is taken as $T=5$. The corresponding Dirichlet boundary conditions are given by the exact solution $u(x,y,t)=D e^{-t}\sin(x+Ct)\cos(y+Ct)$. The boundary treatment strategy for this problem is fully detailed in the supplementary material \ref{app:C}.

\begin{table}[ht]
\caption{Errors and orders of accuracy for PDE $u_t + \partial_x f_1(u) + \partial_y f_2(u) = \partial_x g_1(u_x, u_y) + \partial_y g_2(u_x, u_y) + h(u)$, where $f_1 = -C u$, $f_2 = -C u$, $g_1 = D u_x$, $g_2 = D u_y$, $h = (2D-1)u$, considering $C=0.1$, $D=1$. $CFL=0.2$. Spatial domain $[-1,1]\times[-1,1]$, computing time $T=5$. Dirichlet boundary conditions given by the exact solution $u(x,y,t)=e^{-t}\sin(x+Ct)\cos(y+Ct)$.}
\begin{center}
\begin{tabular}{ccccccc}
    \hline
    $N,M$ & \multicolumn{6}{c}{Without boundary treatment}\\
        \cline{2-7}
        & $L^1$ error & $L^1$ order & $L^2$ error & $L^2$ order & $L^\infty$ error & $L^\infty$ order \\
    \hline
    $5$ & $8.74e-06$ & $-$ & $6.43e-06$ & $-$ & $1.22e-05$ & $-$ \\
    \hline
    $10$ & $1.52e-06$ & $2.52$ & $1.29e-06$ & $2.32$ & $3.15e-06$ & $1.95$ \\
    \hline    
    $20$ & $2.61e-07$ & $2.54$ & $2.63e-07$ & $2.29$ & $8.12e-07$ & $1.96$ \\
    \hline    
    $40$ & $4.50e-08$ & $2.54$ & $5.43e-08$ & $2.28$ & $2.12e-07$ & $1.94$ \\
    \hline    
    $80$ & $7.78e-09$ & $2.53$ & $1.13e-08$ & $2.26$ & $5.45e-08$ & $1.96$ \\
    \hline    
\end{tabular}

\begin{tabular}{ccccccc}
    \hline
    $N$ & \multicolumn{6}{c}{With boundary treatment}\\
        \cline{2-7}
        & $L^1$ error & $L^1$ order & $L^2$ error & $L^2$ order & $L^\infty$ error & $L^\infty$ order \\
    \hline
    $5$ & $4.28e-06$ & $-$ & $2.43e-06$ & $-$ & $2.62e-06$ & $-$ \\
    \hline
    $10$ & $5.35e-07$ & $3.00$ & $2.97e-07$ & $3.03$ & $2.70e-07$ & $3.28$ \\
    \hline    
    $20$ & $6.74e-08$ & $2.99$ & $3.74e-08$ & $2.99$ & $3.22e-08$ & $3.07$ \\
    \hline    
    $40$ & $8.46e-09$ & $2.99$ & $4.71e-09$ & $2.99$ & $4.06e-09$ & $2.99$ \\
    \hline    
    $80$ & $1.06e-09$ & $3.00$ & $5.91e-10$ & $2.99$ & $5.08e-10$ & $3.00$ \\
    \hline    
\end{tabular}

\end{center}
\label{tb-linear2d-sourceterms}
\end{table}

In Table \ref{tb-linear2d-sourceterms}  we show the errors and orders of accuracy for the third-order IMEX LDG scheme, firstly without applying the boundary treatment, and later applying it. This example shows by numerical verification that the designed boundary treatment strategy is also able to recover the third-order accuracy even for multidimensional convection-diffusion-reaction problems with time-dependent Dirichlet boundary conditions.

\subsection{One-dimensional convective heat equation, fourth-order test \label{sec:expO4}}
Finally, this section is devoted to numerically showing that the proposed boundary treatment strategy extends to integrators of higher order than three. Here we solve the linear PDE of Section \ref{subsec:1dheat} with $D=1$.

In this fourth-order numerical experiment, as previously stated, an additional term (up to order one) in the Taylor expansions in time of $u_{xx}^{n,i}$ and $u_{xxx}^{n,i}$ needs to be considered. The mixed time-space derivatives are moved to space derivatives applying the Lax-Wendroff procedure, thus yielding to third, fourth, and fifth-order derivatives. All of them need\footnote{The third order derivative appears also outside this Taylor expansion in time. In such places, it must be approximated at least at second order. We computed it as $(\hat{u}_{xx})_x$, using second-order finite difference approximations of the first derivative (see the previous Remark \ref{remark:finiteDiff}).} to be computed at least at first order since they are multiplied by a constant of $\mathcal{O}(\tau^3)$ in the global boundary treatment strategy and we take $\tau = \mathcal{O}(\Delta x)$. Thus, $\hat{u}_{xxx}$ is directly given by the polynomial solution. The fourth derivative is computed as $(\hat{u}_{xxx})_x$ with the first-order finite difference approximation of the first derivative along the boundary and its closest neighbor cell. Finally, the fifth derivative is computed with the first-order finite difference approximation of the second derivative of the previously computed first-order $\hat{u}_{xxx}$ along the boundary and its two closest neighbor cells.

\begin{figure}[!htb]
\centering
\subfigure{\includegraphics[scale=0.41]{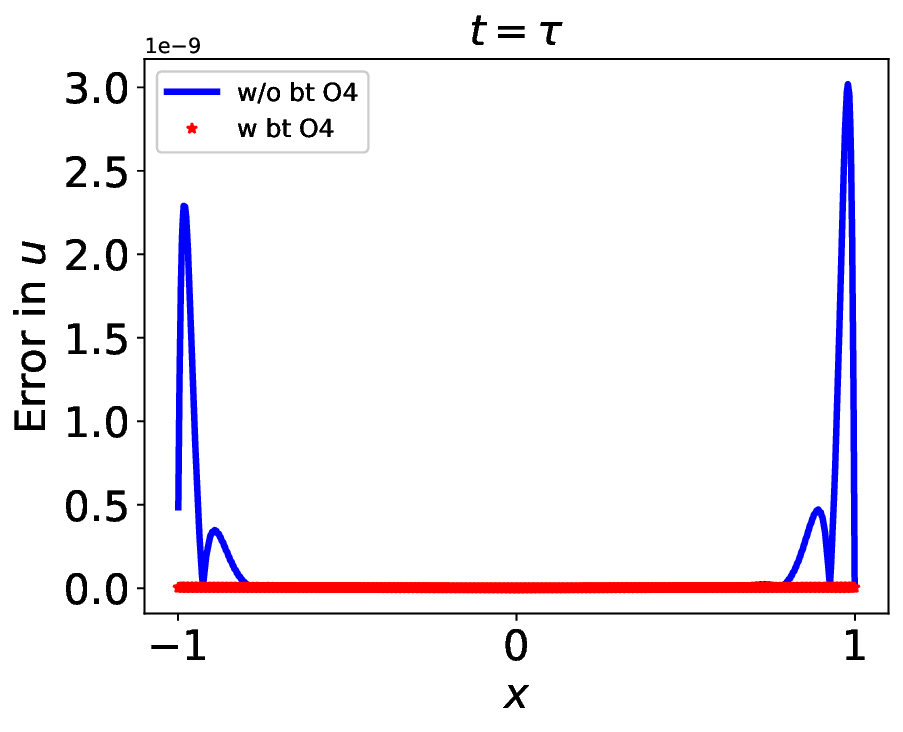}}
\subfigure{\includegraphics[scale=0.41]{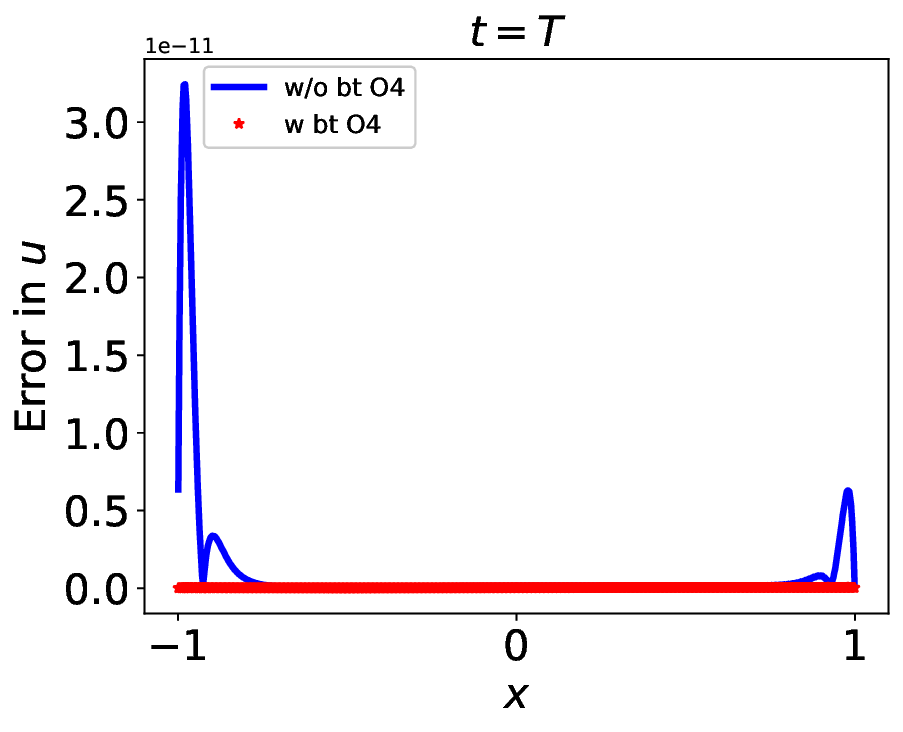}}
\caption{$|\hat{u}-u|$ for $N=160$ with the same data of Table \ref{tb-1d-linear-O4}.}
\label{fig:BL_O4}
\end{figure}

In Table \ref{tb-1d-linear-O4} we show the errors and orders of convergence for the fourth-order IMEX LDG scheme, without and with boundary treatment. Firstly, in the case without boundary treatment, the order reduction is severe (only second $L^\infty$ order). Secondly, when the boundary treatment is applied, fourth-order accuracy is successfully restored for the three considered error norms. Finally, Figure \ref{fig:BL_O4} shows that the boundary treatment strategy vanishes the annoying spatial errors in the regions near the boundaries.

\begin{table}[ht]
\caption{Errors and orders of accuracy for PDE $u_t + \partial_x f(u) = \partial_x g(u_x)$, where $f = -C u$, $g = D u_x$, considering $C=0.1$, $D=1$. $CFL=0.25$. Spatial domain $[-1,1]$, computing time $T=5$. Dirichlet boundary conditions given by the exact solution $u(x,t)=e^{-t} \sin(x+Ct)$.}
\begin{center}
\begin{tabular}{ccccccc}
    \hline
    $N$ & \multicolumn{6}{c}{Without boundary treatment}\\
        \cline{2-7}
        & $L^1$ error & $L^1$ order & $L^2$ error & $L^2$ order & $L^\infty$ error & $L^\infty$ order \\
    \hline
    $5$ & $4.30e-08$ & $-$ & $3.58e-08$ & $-$ & $5.57e-08$ & $-$ \\
    \hline
    $10$ & $4.00e-09$ & $3.42$ & $4.04e-09$ & $3.14$ & $1.08e-08$ & $2.36$ \\
    \hline    
    $20$ & $4.61e-10$ & $3.11$ & $6.91e-10$ & $2.54$ & $2.20e-09$ & $2.29$ \\
    \hline    
    $40$ & $6.89e-11$ & $2.74$ & $1.37e-10$ & $2.33$ & $5.36e-10$ & $2.03$ \\
    \hline    
    $80$ & $1.19e-11$ & $2.53$ & $2.81e-11$ & $2.28$ & $1.32e-10$ & $2.02$ \\
    \hline    
    $160$ & $2.10e-12$ & $2.50$ & $5.83e-12$ & $2.26$ & $3.24e-11$ & $2.02$ \\
    \hline
\end{tabular}

\begin{tabular}{ccccccc}
    \hline
    $N$ & \multicolumn{6}{c}{With boundary treatment}\\
        \cline{2-7}
        & $L^1$ error & $L^1$ order & $L^2$ error & $L^2$ order & $L^\infty$ error & $L^\infty$ order \\
    \hline
    $5$ & $3.61e-08$ & $-$ & $2.96e-08$ & $-$ & $3.96e-08$ & $-$ \\
    \hline
    $10$ & $2.48e-09$ & $3.86$ & $2.05e-09$ & $3.85$ & $2.62e-09$ & $3.92$ \\
    \hline    
    $20$ & $1.57e-10$ & $3.98$ & $1.31e-10$ & $3.97$ & $1.68e-10$ & $3.96$ \\
    \hline    
    $40$ & $9.87e-12$ & $3.99$ & $8.24e-12$ & $3.99$ & $1.08e-11$ & $3.96$ \\
    \hline    
    $80$ & $6.17e-13$ & $4.00$ & $5.17e-13$ & $3.99$ & $6.92e-13$ & $3.96$ \\
    \hline    
    $160$ & $3.90e-14$ & $3.98$ & $3.28e-14$ & $3.98$ & $4.48e-14$ & $3.95$ \\
    \hline 
\end{tabular}
\end{center}
\label{tb-1d-linear-O4} 
\end{table}

\subsection{Boundary treatment efficiency}

Up to here we have shown that boundary treatment is crucial in terms of order of convergence. Here, we aim to emphasize that the computational cost of the boundary strategy is negligible concerning the one of IMEX LDG solvers without boundary treatment. In Figure \ref{fig:efficiency}, accuracy-work diagrams are shown for the numerical experiments previously explained in Sections \ref{subsec:1dheat} and \ref{sec:expO4}. The codes were developed in C++ language and were executed on a recent 13th Gen Intel(R) Core(TM) i7-1365U processor. On the $x$-axis the computing time in seconds is shown. On the $y$-axis, both $L^2$ errors (left) and $L^\infty$ errors (right) are displayed. Both third and fourth-order IMEX LDG methods are shown. Each one of the markers over the lines corresponds to one refinement level in space ($N=5,10,20,40,80,160,320,640$). From the plots, it is clear that the additional cost of the boundary correction technique is negligible relative to a naive application of boundary conditions. In summary, the plots confirm that performing the boundary treatment strategy is highly beneficial.

\begin{figure}[!htb]
\centering
\subfigure{\includegraphics[scale=0.39]{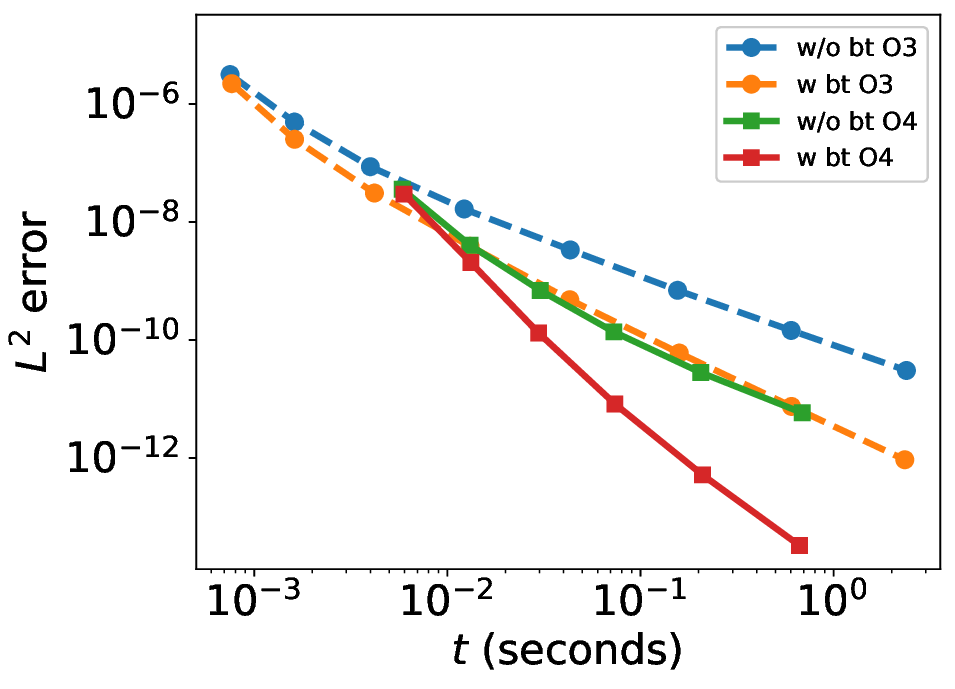}}
\subfigure{\includegraphics[scale=0.39]{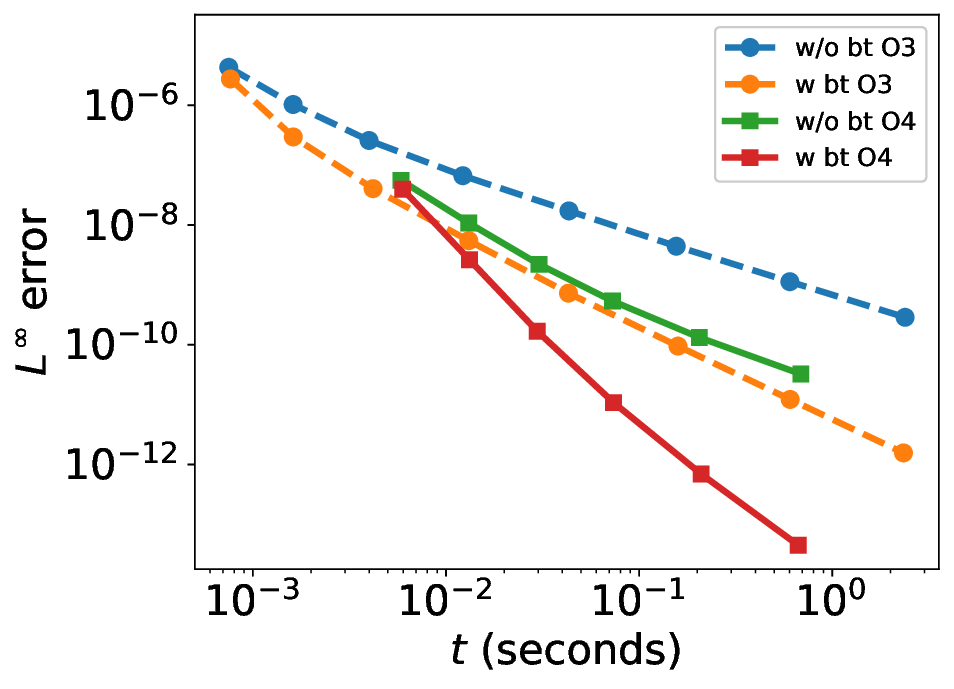}}
\caption{Efficiency plots with the  same data of Table \ref{tb-1d-linear-O4} in the $log_{10}\times log_{10}$ scale.}
\label{fig:efficiency}
\end{figure}

\section{Conclusions}
In this article, we propose novel algorithms to overcome the order reduction phenomenon when IMEX RK methods are applied to PDE problems with time-dependent Dirichlet boundary conditions. The algorithms treat boundary points in the same way as interior points. To this end, a Cauchy-Koval\'evskaya procedure and the differentiation of the IMEX internal stages are carried out. The proposed algorithms can be applied to general IMEX schemes and nonlinear PDE problems in dimensions greater than one. Besides, the strategies can be directly applied to problems with possibly stiff source terms. We show numerically that the designed algorithms can recover the desired order of accuracy for one-dimensional linear and nonlinear problems, as well as two-dimensional problems. Finally, we would like to point out that our boundary treatment algorithms can be directly applied to explicit schemes since explicit schemes are a special case of IMEX RK time integration methods.

\section*{Acknowledgments}
We would like to express our gratitude to anonymous reviewers for their valuable comments which helped us to improve the quality of the work.

\bibliographystyle{siamplain}
\bibliography{references}

\clearpage

{\begin{center}
\textbf{Supplementary materials} for the article \textit{``Boundary treatment for high-order IMEX Runge-Kutta Local Discontinuous Galerkin schemes for multidimensional nonlinear parabolic PDEs''}
 \end{center}
}
\setcounter{section}{0}
 \setcounter{page}{1}
\renewcommand*{\thesection}{\Alph{section}}

\vspace{0.2cm}
The supplementary materials explicitly detail the boundary treatment procedure developed in Section \ref{sec:bt_gt}, for the PDEs numerically solved in Section \ref{sec:numericalExperiments}, considering the third-order IMEX scheme \eqref{eq:imex1}-\eqref{eq:imexq12_123}. Therefore, this section is intended to allow the reader to easily reproduce the numerical experiments of our article.

\section{One-dimensional linear PDE with source terms} \label{app:A} In this section, the boundary treatment strategy is applied to the following one-dimensional PDE presented in Section \ref{subsec:1dheat} of the article,
\begin{equation*}
    u_t + \partial_x f(u) = \partial_x g(u_x) + h(u),
\end{equation*}
where
\begin{align*}
    f(u) = -C u, \quad
    g (u_x) = D u_x, \quad
    h(u) = (D-1)u,
\end{align*}
$C,D\in\R$ are the convection and diffusion coefficients, respectively. Besides, a ti\-me-de\-pen\-dent boundary condition $u_{\lvert\Gamma} = \omega(t)$ is considered.
Therefore, 
\begin{align}
    u_t =& C u_x {+ (D-1)u} +D u_{xx} .    \label{eq:edp_1d_lineal}
\end{align}
On the one hand, for $D=1$, this PDE is the one presented in \cite{WANG2018164}. On the other hand, for $D\neq 1$, the PDE has a linear source term. Note that $u(x,t)=e^{-t}\sin(x+Ct)$ is the exact solution of the problem.

\subsection{First IMEX stage}
The first IMEX stage \eqref{eq:IMEX_1-general-1d} in this case, is
\begin{align}
    \frac{u^{n,1} - u^n}{\gamma \tau}=&  \left[C u_x^n {+(D-1)u^n}  \right]+ {\left[ D u^{n,1}_{xx} \right]}.   \label{eq:IMEX_1-linear-1d}
\end{align}
According \eqref{eq:general-explicit-part-1d}, the explicit part is given by
\begin{equation*}
    \xi^{n,i}(u,x,t) = Cu^{n,i}_x +(D-1)u^{n,i}.
\end{equation*}
Then, equation \eqref{eq:IMEX_1_1-general-1d} reads
\begin{align}
    \frac{u^{n,1} - u^n}{\gamma \tau}=&  C u_x^{n,0} +(D-1)u^{n,0} + u^{n,1}_t - \left( C u_x^{n,1} +(D-1)u^{n,1}\right).  \label{eq:IMEX_1_1-linear-1d}
\end{align}
Now, $u_x^{n,1}$ is computed with \eqref{eq:uxn1-general-1d}, substituting and regrouping, the equation reduces to
\begin{align*}
     u^{n,1} - u^n = \gamma \tau {\big[}u_t^{n,1} {-(D-1)(u^{n,1}-u^n)} {\big]} - \gamma^2 \tau^2 C \l {(D-1)u_x^n+}C u_{xx}^n  +  D u^{n,1}_{xxx} \r. 
\end{align*}
At this point, we apply Proposition \ref{approximation_tnl}, up to the first term to approximate $u_{xxx}^{n,1}$. We end up with
\begin{align*}
u^{n,1} - u^n =& \gamma \tau {\big[}u_t^{n,1} {-(D-1)(u^{n,1}-u^n)} {\big]} - \gamma^2 \tau^2 C \l {(D-1)u_x^n+} C u_{xx}^n  +  D u^{n}_{xxx} \r.
\end{align*}
Finally, we need to solve the linear equation, we propose:
\begin{enumerate}
\item If we solve the equation analytically, which in this case is easy to solve for $u^{n,1},$  the value at the first stage can finally be written as
\begin{align*}
u^{n,1} =& u^n + {\frac{1}{1+\gamma\tau(D-1)}} {\Big[} \gamma \tau u_t^{n,1} - \gamma^2 \tau^2 C \l {(D-1)u_x^n+} C u_{xx}^n  +  D u^{n}_{xxx} \r {\Big]}. 
\end{align*}

\item If we approximate the term $u^{n,1}$ of the right-hand side to time $t^n$ considering \eqref{eq:approx_nl_a_n0} with the first two terms of the expansion, then we get:
\begin{align*}
u^{n,1} =& 
 u^n + \gamma \tau u_t^{n,1} + \gamma \tau (D-1)u^n - \gamma \tau (D-1) \l{u^{n} + (t^n + \gamma \tau - t^n)u_t^n + O((\gamma \tau)^2)} \r \\ 
&- \gamma^2 \tau^2 C \l (D-1)u_x^n+ C u_{xx}^n  +  D u^{n}_{xxx} \r  \\
=&u^n + \gamma \tau u_t^{n,1} - \gamma^2 \tau^2 (D-1) {u_t^n } \\
&- \gamma^2 \tau^2 C \l (D-1)u_x^n+ C u_{xx}^n  +  D u^{n}_{xxx} \r + \mathcal{O}(\tau^3).
\end{align*}

\end{enumerate}

\subsection{Second IMEX stage}
The second IMEX stage \eqref{IMEX-stage2-general-1d} is now given in this case by:
\begin{align*}
    \frac{u^{n,2}-u^n}{\tau} =& \l \frac{1 + \gamma}{2} - \alpha_1\r \left[ C u_x^n {+(D-1)u^n}\right] + \alpha_1 \left[ C u_x^{n,1} {+(D-1)u^{n,1}}  \right] \\
    & + \frac{1 - \gamma}{2}\left[ D u^{n,1}_{xx}  \right] + \gamma \left[ D u^{n,2}_{xx} \right].
\end{align*}
Thus, applying the general equation \eqref{eq:IMEX_2_2-general-1d}, we obtain:
\begin{align}
     \frac{u^{n,2}-u^n}{\tau} =&  {\alpha_1} u_t^{n,1} + \gamma u_t^{n,2} + \l \frac{1 - \gamma}{2} -\alpha_1 \r \frac{u^{n,1} - u^n}{\gamma \tau}   \nonumber \\
      & + \gamma \left[ C u_x^n {+(D-1)u^n} \right] + \gamma \left[ -C u_x^{n,2} {-(D-1)u^{n,2}} \right]. \label{eq:IMEX_2_2-linear-1d}
\end{align}
Therefore, now we are going to calculate the derivative of $u^{n,2}$ using the IMEX scheme. Note that from these third-order approximations of the derivatives in $\tau$ and applying Section \ref{subsec:IMEX_gen_stage2}, we know that
\begin{equation*}
    \frac{u_{x}^{n,2} - u_{x}^n}{\tau} = \l \frac{1 + \gamma}{2}\r \frac{u_{x}^{n,1} - u_{x}^n}{\gamma\tau},
\end{equation*}
being $\frac{u_{x}^{n,1} - u_{x}^n}{\gamma\tau}$ the replacement we did in the first stage. Therefore:
\begin{align*}
    \frac{u^{n,2}-u^n}{\tau} =&  {\alpha_1} u_t^{n,1} + \gamma u_t^{n,2} {+\gamma (-(D-1)) (u^{n,2}-u^n)} + \l \frac{1 - \gamma}{2} -\alpha_1 \r \frac{u^{n,1} - u^n}{\gamma \tau}  \\
    & -C  \l  \frac{1 + \gamma}{2}\r \l u_{x}^{n,1} - u_{x}^n \r,
\end{align*}
where $u_x^{n,1}- u_x^{n}$ is calculated with \eqref{eq:uxn1-general-1d}. Then, the linear equation for this stage reads
\begin{align*}
    u^{n,2}-u^n =& \alpha_1 \tau u_t^{n,1} + \gamma \tau u_t^{n,2}  + \l \frac{1 - \gamma}{2} -\alpha_1 \r \frac{u^{n,1} - u^n}{\gamma}  -C  \l  \frac{1 + \gamma}{2}\r \tau \l u_{x}^{n,1} - u_{x}^n \r \\
    & +\gamma \tau (-(D-1)) ({u^{n,2}}-u^n). 
\end{align*}
We propose the following ways to solve it:
\begin{enumerate}
\item Solving it analytically we get the following expression for $ u^{n,2} $:
\begin{align*}
    u^{n,2} =& u^n + \frac{\tau}{1+\tau\gamma(D-1)}\Bigg[ {\alpha_1} u_t^{n,1} + \gamma u_t^{n,2} + \l \frac{1 - \gamma}{2} -\alpha_1 \r \frac{u^{n,1} - u^n}{\gamma \tau}  \\
    & \hspace{3.5cm} -C  \l  \frac{1 + \gamma}{2}\r \l u_{x}^{n,1} - u_{x}^n \r \Bigg].
\end{align*}

\item Passing the term $u^{n,2}$ of the right hand side to time $t^n$ by using \eqref{eq:approx_nl_a_n0}, we obtain:
\begin{align*}
    u^{n,2} =&  
    u^n + \alpha_1 \tau u_t^{n,1} + \gamma \tau u_t^{n,2}  + \l \frac{1 - \gamma}{2} -\alpha_1 \r \frac{u^{n,1} - u^n}{\gamma}   -C  \l  \frac{1 + \gamma}{2}\r \tau \l u_{x}^{n,1} - u_{x}^n \r \\
    & \hspace{1cm} +\gamma \tau (-(D-1)) { \l {u^{n}} +  \frac{1+\gamma}{2} \tau u_t^n \r} + \gamma \tau (D-1) {u^n} \\
    =& u^n + \alpha_1 \tau u_t^{n,1} + \gamma \tau u_t^{n,2}  + \l \frac{1 - \gamma}{2} -\alpha_1 \r \frac{u^{n,1} - u^n}{\gamma}  -C  \l  \frac{1 + \gamma}{2}\r \tau \l u_{x}^{n,1} - u_{x}^n \r \\
    & \hspace{1cm} -\gamma \tau^2 (D-1) {   \frac{1+\gamma}{2}  u_t^n }.       
\end{align*}
\end{enumerate}

\subsection{Third IMEX stage}
The third IMEX stage \eqref{eq:imexStep3-general-1d} reads
\begin{align*}
    \frac{u^{n,3} - u^n}{\tau} =& \l 1 -\alpha_2 \r \left[ C u_x^{n,1} {+(D-1)u^{n,1}}\right] + \alpha_2 \left[ C u_x^{n,2} {+(D-1)u^{n,2}}\right] \\
    &+ \beta_1 \left[ D u^{n,1}_{xx} \right] + \beta_2 \left[ D u^{n,2}_{xx} \right] +\gamma \left[ D u^{n,3}_{xx} \right].
\end{align*}
Following Section \ref{subsec:IMEX_gen_stage3} this results in
\begin{align*}
    \frac{u^{n,3} - u^n}{\tau} = & \l {1-\alpha_2} \r u_t^{n,1} + \beta_2u_t^{n,2} + \gamma u_t^{n,3} +\l \alpha_2 + \beta_1 -1 \r \frac{u^{n,1} - u^n}{\gamma \tau}  \\ 
    &+ \l \beta_2 - \alpha_2 \r \left[  \frac{u^{n,2}-u^n}{\gamma \tau} - \frac{{\alpha_1}}{\gamma } u_t^{n,1} -  u_t^{n,2} - \l \frac{1 - \gamma}{2} -\alpha_1 \r \frac{u^{n,1} - u^n}{{\gamma^2} \tau} \right] \\
    &+\gamma \left[ C u_x^n {+(D-1)u^{n}} \right] +\gamma \left[  -C u_x^{n,3} {-(D-1)u^{n,3}} \right]. 
\end{align*}
It remains to evaluate the directional derivatives of $u^{n,3}$, we use \eqref{eq:approx_x_n3_x_n}, so that:
\begin{align*}
    \frac{u^{n,3} - u^n}{\tau} =&  \l {1-\alpha_2}  \r u_t^{n,1} + \beta_2u_t^{n,2} + \gamma u_t^{n,3} +\l \alpha_2 + \beta_1 -1 \r \frac{u^{n,1} - u^n}{\gamma \tau}  \\ 
    &+ \l \beta_2 - \alpha_2 \r \left[  \frac{u^{n,2}-u^n}{\gamma \tau} - \frac{{\alpha_1}}{\gamma } u_t^{n,1} -  u_t^{n,2} - \l \frac{1 - \gamma}{2} -\alpha_1 \r \frac{u^{n,1} - u^n}{ {\gamma^2}\tau} \right] \\
    &+\gamma \left[ C { u_x^n} {+(D-1)u^{n}} \right] \\
    &+\gamma \left[  -C \bigg( {u_x^{n}} + \frac{u_x^{n,1}-u_x^{n}}{\gamma} \bigg) {-(D-1)u^{n,3}} \right] \nonumber \\
    =&  \l {1-\alpha_2} - \frac{(\beta_2-\alpha_2)\alpha_1}{\gamma}  \r u_t^{n,1} + \alpha_2u_t^{n,2} + \gamma u_t^{n,3} {+ \gamma(-(D-1))(u^{n,3}-u^n)} \\ 
    & + \l \alpha_2 + \beta_1 -1 - \frac{\beta_2-\alpha_2}{\gamma} \l \frac{1-\gamma}{2}-\alpha_1\r \r \frac{u^{n,1} - u^n}{\gamma \tau} \nonumber \\
    &+ \frac{ \beta_2 - \alpha_2 }{\gamma}  \left[  \frac{u^{n,2}-u^n}{ \tau}  \right]   -C  \l u_x^{n,1}-u_x^{n} \r.  \nonumber
\end{align*}
Again, we need to solve the linear equation. As before, here we propose two ways:
\begin{enumerate}
\item The algebraic solution of the linear equation:
\begin{align*}
    u^{n,3} =& u^n + \frac{\tau}{1+\gamma\tau(D-1)} \Bigg\lbrace  \l 1-\alpha_2 - \frac{(\beta_2-\alpha_2)\alpha_1}{\gamma}  \r u_t^{n,1} + \alpha_2u_t^{n,2} + \gamma u_t^{n,3} \\ 
    &\hspace{3.3cm} + \left[ \alpha_2 + \beta_1 -1 - \frac{\beta_2-\alpha_2}{\gamma} \l \frac{1-\gamma}{2}-\alpha_1\r \right] \frac{u^{n,1} - u^n}{\gamma \tau} \nonumber \\
    &\hspace{3.3cm} + \frac{ \beta_2 - \alpha_2 }{\gamma}  \left[  \frac{u^{n,2}-u^n}{ \tau}  \right] -C  \l u_x^{n,1}-u_x^{n} \r  \Bigg\rbrace.
\end{align*}

\item Taking $u^{n,3}$ on the right hand side to time $t^n$ leads to
\begin{align*}
    u^{n,3} - u^n =&  \l 1-\alpha_2 - \frac{(\beta_2-\alpha_2)\alpha_1}{\gamma}  \r \tau u_t^{n,1} + \alpha_2 \tau u_t^{n,2} + \gamma \tau u_t^{n,3}  \\ 
    & + \l \alpha_2 + \beta_1 -1 - \frac{\beta_2-\alpha_2}{\gamma} \l \frac{1-\gamma}{2}-\alpha_1\r \r \frac{u^{n,1} - u^n}{\gamma } \nonumber \\
    &+ \frac{ \beta_2 - \alpha_2 }{\gamma}  \left[  u^{n,2}-u^n  \right]  -C \tau  \l u_x^{n,1}-u_x^{n} \r  - \gamma \tau^2 (D-1){u_t^n}.   \nonumber      
\end{align*}

\end{enumerate}

\section{One-dimensional nonlinear PDE with source terms}\label{app:B}

In this section, the proposed boundary treatment strategy is applied to the following diffusive Burgers equation with source terms solved in Section \ref{subsec:1dburgers} of the article,
\begin{equation*}
    u_t + \partial_x f(u) = \partial_x g(u_x) + h(u,x,t),
\end{equation*}
where
\begin{equation*}
    f(u) = \frac{u^2}{2}, \quad
    g(u_x) = D u_x, \quad
    h(u,x,t) = p(x,t)u,
\end{equation*}
being $D\in\R$ the diffusion parameter. Therefore, 
\begin{align}
    u_t =& - u u_x +D u_{xx} + p(x,t)u.    \label{eq:burgers-with-source}
\end{align}

\subsection{First IMEX stage}
The first IMEX stage \eqref{eq:IMEX_1-general-1d} reads
\begin{align*}
    \frac{u^{n,1} - u^n}{\gamma \tau}=&  {\left[-u^n u_x^n +p(x,t)u^n \right]}+ {\left[ D u^{n,1}_{xx} \right]}.   
\end{align*}
According to  \eqref{eq:general-explicit-part-1d}, the explicit part operator is
\begin{equation*}
        \xi^{n,i}(x) = -u^{n,i} u_x^{n,i} +p(x,t^{n,i})u^{n,i}. 
\end{equation*}
Substituting the derivative \eqref{eq:uxn1-general-1d} on \eqref{eq:IMEX_1_1-general-1d} for this particular case, and replacing $u_{xxx}^{n,1}$ with $u_{xxx}^{n,0}$ we end up with

\begin{align*}
 u^{n,1} =& u^n + \gamma\tau \l u^{n,1}_t + p(x,t)u^n  - u^n u_x^n\r -\gamma\tau p(x,t^{n,1}){u^{n,1}} \\
 &+ \gamma\tau {u^{n,1}} \Bigg[ u_x^{n} +\gamma\tau\Big( - (u^n_x)^2 - u^nu^n_{xx}+ p'(x,t^n)u^n + p(x,t^n)u_x^n  +  D u^{{n}}_{xxx} \Big)  \Bigg].
 \end{align*}
In order to deal with this nonlinear equation in $u^{n,1}$ we propose:
\begin{enumerate}
\item Obtaining the exact solution:
$$ \hspace{-0.2cm}
u^{n,1} = \frac{u^n +\gamma \tau \l -u^n u_x^n +p(x,t^n)u^n +  u^{n,1}_t\r}{1 - \gamma^2\tau^2 \l -\frac{p(x,t^{n,1})}{\gamma \tau} + \frac{ u_x^{n}}{\gamma\tau} - (u^n_x)^2 - u^nu^n_{xx} + p'(x,t^n)u^n + p(x,t^n)u^n_x  +  D u^{n}_{xxx} \r } .
$$

\item Solving by transforming the terms $u^{n,1}$ in the right side to terms in $u^{n,0}$ according to the approximation obtained in \eqref{eq:approx_nl_a_n0}:
\begin{align*}
 u^{n,1} =& 
 u^n + \gamma\tau \l u^{n,1}_t + p(x,t)u^n  - u^n u_x^n\r \\
 &+ \gamma\tau {[u^{n}+\gamma \tau u_t^n]} \Bigg[-p(x,t^{n,1}) +u_x^{n} \\
 & \hspace{1cm}+\gamma\tau\Big( - (u^n_x)^2 - u^nu^n_{xx}+ p'(x,t^n)u^n + p(x,t^n)u_x^n  +  D u^{{n}}_{xxx} \Big)  \Bigg]. 
\end{align*}

\end{enumerate}

\subsection{Second IMEX stage}
In this case, the second IMEX stage \eqref{IMEX-stage2-general-1d} reads:
\begin{align*}
    \frac{u^{n,2}-u^n}{\tau} =& \l \frac{1 + \gamma}{2} - \alpha_1\r \left[ -u^n u_x^n +p(x,t^n)u^n \right] + \alpha_1 \left[ -u^{n,1} u_x^{n,1} + p(x,t^{n,1})u^{n,1}  \right] \\
    & + \frac{1 - \gamma}{2}\left[ D u^{n,1}_{xx}  \right] + \gamma \left[ D u^{n,2}_{xx} \right].
\end{align*}
As previously, we substitute \eqref{eq:IMEX_2_2-u_x} on \eqref{eq:IMEX_2_2-general-1d} for this particular case, leading to
\begin{align*}
     u^{n,2} =& u^n + \alpha_1\tau u_t^{n,1} + \gamma \tau u_t^{n,2} + \l \frac{1 - \gamma}{2} -\alpha_1 \r \frac{u^{n,1} - u^n}{\gamma}  + \gamma \tau \left[ -u^n u_x^n +p(x,t^n)u^n\right] \nonumber \\
    &+ \gamma \tau  u^{n,2} \l u_x^n + \l\frac{1 + \gamma}{2}\r \dfrac{u_x^{n,1}-u_x^n}{\gamma} -p(x,t^{n,2})\r.
\end{align*}

Now it is needed to solve the nonlinear equation, we proceed in two ways:
\begin{enumerate}

\item Obtaining the analytical solution for the denominator being not null:
$$
    \hspace{-0.5cm}u^{n,2} =  \frac{u^n +  \tau \l {\alpha_1} u_t^{n,1} + \gamma u_t^{n,2} + \l \frac{1 - \gamma}{2} -\alpha_1 \r \frac{u^{n,1} - u^n}{\gamma \tau}     + \gamma \left[ \xi^n \right] \r}{1 + \gamma \tau p(x,t^{n,2}) - \gamma^2\tau^2 \l u_x^n/\gamma\tau + \l   \frac{1 + \gamma}{2} - \alpha_1\r \left[ \xi_x^{n} \right]  + \alpha_1 \left[\xi_x^{n,1} \right] + \frac{1 + \gamma}{2}\left[  D u^{n}_{xxx}  \right]\r}.
$$

\item If we approximate $u^{n,2}$, in terms of $u^n$ as described in \eqref{eq:approx_nl_a_n0}:
\begin{align*}
     {u^{n,2}} =& u^n + {\alpha_1} \tau u_t^{n,1} + \gamma \tau u_t^{n,2} + \l \frac{1 - \gamma}{2} -\alpha_1 \r \frac{u^{n,1} - u^n}{\gamma}  + \gamma \tau \left[ -u^n u_x^n +p(x,t^n)u^n\right] \nonumber \\
    &+ \gamma \tau  {\left[u^{n} + \dfrac{1+\gamma}{2}\tau u_t^n + \mathcal{O}(\tau^2)\right]} \l u_x^n + \l\frac{1 + \gamma}{2}\r \dfrac{u_x^{n,1}-u_x^n}{\gamma} -p(x,t^{n,2})\r.
\end{align*}

\end{enumerate}

\subsection{Third IMEX stage}
Following the procedure described in Section \ref{subsec:IMEX_gen_stage3}, equation \eqref{eq:IMEX_3_general} with the substitution \eqref{eq:approx_x_n3_x_n} is in this case

\begin{align*}
    u^{n,3} =& u^n + \l 1-\alpha_2 - \frac{(\beta_2-\alpha_2)\alpha_1}{\gamma}  \r \tau u_t^{n,1} + \alpha_2 \tau u_t^{n,2} + \gamma \tau u_t^{n,3}  \\ 
    & + \l \alpha_2 + \beta_1 -1 - \frac{\beta_2-\alpha_2}{\gamma} \l \frac{1-\gamma}{2}-\alpha_1\r \r \frac{u^{n,1} - u^n}{\gamma } + \frac{ \beta_2 - \alpha_2 }{\gamma}  \left[ u^{n,2}-u^n  \right]\nonumber \\
    & - \gamma \tau [u^n u_x^n -p(x,t^n)u^n]+ \gamma \tau u^{n,3} \bigg( u_x^{n} + \frac{u_x^{n,1}-u_x^{n}}{\gamma} - p(x,t^{n,3})\bigg).
\end{align*}

This nonlinear equation has to be solved. Again, the two ways we propose are:
\begin{enumerate}
\item Solve the nonlinear equation for $u^{n,3}$:
\begin{equation*}
    u^{n,3} =   \frac{ u^{n} + \tau \cdot Num}{Den},
\end{equation*}
where
\begin{align*}
    Num = &\l {1-\alpha_2} \r u_t^{n,1} + \beta_2u_t^{n,2} + \gamma u_t^{n,3} +\l \alpha_2 + \beta_1 -1 \r \frac{u^{n,1} - u^n}{\gamma \tau} \\
    &+ \l \beta_2 - \alpha_2 \r \left[  \frac{u^{n,2}-u^n}{\gamma \tau} - \frac{{\alpha_1}}{\gamma } u_t^{n,1} -  u_t^{n,2} - \l \frac{1 - \gamma}{2} -\alpha_1 \r \frac{u^{n,1} - u^n}{{\gamma^2} \tau} \right] \\
    & +\gamma \left[ \xi^n \right], \\
    Den =&  1 + \gamma \tau p(x,t^{n,3}) - \gamma \tau^2  \left[u_x^n/\tau +\l 1 -\alpha_2 \r \xi_x^{n,1} + \alpha_2 \xi_x^{n,2}  +   D u^{n}_{xxx}  \right],
\end{align*}
for $Den\neq 0$.
\item Approximate $u^{n,3}$ of the right hand side in terms of $u^{n,0}$ ending up with
\begin{align*}
    {u^{n,3}} =& u^n + \l 1-\alpha_2 - \frac{(\beta_2-\alpha_2)\alpha_1}{\gamma}  \r \tau u_t^{n,1} + \alpha_2 \tau u_t^{n,2} + \gamma \tau u_t^{n,3}  \\ 
    & + \l \alpha_2 + \beta_1 -1 - \frac{\beta_2-\alpha_2}{\gamma} \l \frac{1-\gamma}{2}-\alpha_1\r \r \frac{u^{n,1} - u^n}{\gamma } + \frac{ \beta_2 - \alpha_2 }{\gamma}  \left[ u^{n,2}-u^n  \right] \\
    & - \gamma \tau [u^n u_x^n -p(x,t^n)u^n]+ \gamma \tau {\left[ u^{n} +\tau u_t^n\right]} \bigg( u_x^{n} + \frac{u_x^{n,1}-u_x^{n}}{\gamma} - p(x,t^{n,3})\bigg).  \nonumber    
\end{align*}
\end{enumerate}
As before, the second approach is numerically preferable, since there is no need to be alert with zero denominators in algebraic solutions.

\section{Two-dimensional convective heat equation with source term} \label{app:C}

Finally, we show the detailed boundary treatment strategy for the following bidimensional heat equation solved in Section \ref{sec:exp2d} of the article,
\begin{equation}\label{eq:2d_linear}
    u_t + \partial_x f_1(u) + \partial_y f_2(u) = \partial_x g_1(u_x, u_y) + \partial_y g_2(u_x, u_y) + h(u),
\end{equation}
where $$f_1 = -Cu, \quad f_2 = -Cu, \quad g_1 = D u_x, \quad g_2 = Du_y, \quad h = (2D-1)u,$$
$C,D\in\R$. Therefore,
\begin{equation} \label{eq:pdeHeat2d}
    u_t = C u_x + C u_y  + (2D-1)u + D u_{xx} + D u_{yy}.
\end{equation}
Note that $u(x,y,t)=e^{-t}\sin(x+Ct)\cos(y+Ct)$ is the exact solution of this problem.

Initially, the first, second, and third IMEX stages are fully detailed in subsections \ref{sm:2d_1st}, \ref{sm:2d_2nd} and \ref{sm:2d_3rd}, respectively.  Later, having in mind that third-order mixed derivatives in space need to be computed in the resulting strategies of the previous subsections, a proper way to approximate them is presented in subsection \ref{sm:2d_4th}. Finally, some plots of the numerical solutions are presented in subsection \ref{sm:2d_5th}.

\subsection{First IMEX stage} \label{sm:2d_1st}
The first IMEX stage reads
\begin{align}
    \frac{u^{n,1} - u^n}{\gamma \tau}=&  \left[C u_x^n + C u_y^n +  \l 2D - 1 \r  u^n \right]+ {\left[ D u^{n,1}_{xx} + D u^{n,1}_{yy}  \right]}.   \label{eq:IMEX_1}
\end{align}
We assume that the PDE \eqref{eq:pdeHeat2d} is fulfilled at time $t^{n,1}$, therefore:
\begin{align*}
     D u^{n,1}_{xx} + D u^{n,1}_{yy} & = u^{n,1}_t - C u^{n,1}_x - C u^{n,1}_y  - (2D-1)u^{n,1}.
\end{align*}
Substituting the diffusive term on the first stage \eqref{eq:IMEX_1}:
\begin{align}
    \frac{u^{n,1} - u^n}{\gamma \tau}=& \left[Cu^n_x + C u^n_y + (2D-1)u^n \right] + u^{n,1}_t - Cu^{n,1}_x - Cu^{n,1}_y - (2D-1)u^{n,1}. \label{eq:IMEX_1_1}
\end{align}
Now we compute the directional derivatives of $u^{n,1}$ by taking derivatives on equation \eqref{eq:IMEX_1} respect to $x$ and $y$
\begin{align}
   \frac{ u_x^{n,1}}{\gamma\tau} =&  \frac{ u_x^{n}}{\gamma\tau} + \left[C u^n_{xx} + C u^n_{yx} + (2D-1)u_x^n \right] +  D u^{n,1}_{xxx} + D u^{n,1}_{yyx},  \label{eq:uxn1}\\
   \frac{ u_y^{n,1}}{\gamma\tau} =&  \frac{ u_y^{n}}{\gamma\tau} + \left[C u^n_{xy} + C u^n_{yy} + (2D-1)u_y^n \right] + D u_{xxy}^{n,1} +D u^{n,1}_{yyy}. \label{eq:uyn1}
\end{align}
Substituting these derivatives on \eqref{eq:IMEX_1_1} we end up with
\begin{align*}
     u^{n,1} - u^n  =& 
   \gamma\tau \left[u^{n,1}_t - \l 2D-1\r \l  u^{n,1} -u^n \r \right] \\
   & - \gamma^2 \tau^2 C   \bigg\lbrace  C u^n_{xx} + C u^n_{yx} + (2D-1)u_x^n  + D u^{n,1}_{xxx} +D u^{n,1}_{yyx} \bigg\rbrace \\
   &  -  \gamma^2 \tau^2 C \bigg\lbrace  C u^n_{xy}  +  C u^n_{yy}  + (2D-1)u_y^n  + D u^{n,1}_{xxy} +D  u^{n,1}_{yyy} \bigg\rbrace.
\end{align*}
Now we apply Propositition \ref{approximation_tnl} to compute the partial derivatives $\partial_{x,y}^{(k)} u^{n,l} = \partial_{x,y}^{(k)} u^{n} + \mathcal{O}(\tau)$. As a result
\begin{align*}
u^{n,1} - u^n    &= \gamma\tau \left[u^{n,1}_t - \l 2D-1\r \l  u^{n,1} -u^n \r \right]  \\
     &\quad- \gamma^2 \tau^2 C   \bigg\lbrace  C u^n_{xx} + C u^n_{yx} + (2D-1)u_x^n  + D u^{n}_{xxx} +D u^{n}_{yyx} \bigg\rbrace \\
   & \quad -  \gamma^2 \tau^2 C \bigg\lbrace  C u^n_{xy}  +  C u^n_{yy}  + (2D-1)u_y^n  + D u^{n}_{xxy} +D  u^{n}_{yyy} \bigg\rbrace.
\end{align*}
Finally, to solve this equation, once again we propose:
\begin{enumerate}
\item {Analytical solution for $u^{n,1}$, the final formula is}
\begin{align*}
u^{n,1} =& u^n + \dfrac{1}{1+\gamma\tau (2D-1)} 
   \Bigg[\gamma\tau u^{n,1}_t  \\
   & \hspace{1cm} - \gamma^2 \tau^2 C  \bigg\lbrace C u^n_{xx} + C u^n_{yx} + (2D-1)u_x^n  + D u^{n}_{xxx} +D u^{n}_{yyx} \bigg\rbrace \nonumber \\
   &\hspace{1cm} - \gamma^2 \tau^2 C \bigg\lbrace C u^n_{xy}  +  C u^n_{yy}  + (2D-1)u_y^n  + D u^{n}_{xxy} +D  u^{n}_{yyy}   \bigg\rbrace \Bigg]. 
\end{align*}
 
 \item Application of Proposition \ref{approximation_tnl} to approximate $u^{n,1}$ with $u^{n,0}$, then
\begin{align*}
u^{n,1} &= u^n +\gamma\tau u^{n,1}_t - \gamma^2\tau^2 \l 2D-1\r   { u_t^n }  \\
     &\quad- \gamma^2 \tau^2 C   \bigg\lbrace  C u^n_{xx} + C u^n_{yx} + (2D-1)u_x^n  + D u^{n}_{xxx} +D u^{n}_{yyx} \bigg\rbrace \\
   & \quad -  \gamma^2 \tau^2 C \bigg\lbrace  C u^n_{xy}  +  C u^n_{yy}  + (2D-1)u_y^n  + D u^{n}_{xxy} +D  u^{n}_{yyy} \bigg\rbrace.      
\end{align*}

\end{enumerate}

\subsection{Second IMEX stage} \label{sm:2d_2nd}
The second IMEX stage for the two-dimensional case reads
\begin{align*}
    \frac{u^{n,2}-u^n}{\tau} =& \l \frac{1 + \gamma}{2} - \alpha_1\r \left[ C u_x^n + C u_y^n + (2D-1)u^n \right]  + \alpha_1 \left[ C u_x^{n,1} + C u_y^{n,1} + (2D-1)u^{n,1}  \right] \\
    &+ \frac{1 - \gamma}{2}\left[ D u^{n,1}_{xx} + D u^{n,1}_{yy} \right] + \gamma \left[ D u^{n,2}_{xx} + D u^{n,2}_{yy}  \right].
\end{align*}
Following Section \ref{subsec:IMEX_gen_stage2}, $u^{n,2}$ can be expressed as
\begin{align*}
    \frac{u^{n,2}-u^n}{\tau} 
    =&  {\alpha_1} u_t^{n,1} + \gamma u_t^{n,2} + \gamma \l -(2D-1)\r \l u^{n,2} - u^n\r + \l \frac{1 - \gamma}{2} -\alpha_1 \r \frac{u^{n,1} - u^n}{\gamma \tau} \nonumber \\
    &  -C  \l  \frac{1 + \gamma}{2}\r ({u_{x}^{n,1} - u_{x}^n})   -C \l  \frac{1 + \gamma}{2}\r ({u_{y}^{n,1} - u_{y}^n}).
\end{align*}
Finally, to deal with this linear equation we propose:
\begin{enumerate}
\item Analytical solution:
\begin{align*}
    u^{n,2}
    = u^n + \frac{\tau}{1+\gamma \tau \l 2D-1\r} \Bigg[& {\alpha_1} u_t^{n,1} + \gamma u_t^{n,2}  + \l \frac{1 - \gamma}{2} -\alpha_1 \r \frac{u^{n,1} - u^n}{\gamma \tau} \nonumber \\
    &  \hspace{-0.11cm}-C  \l  \frac{1 + \gamma}{2}\r ({u_{x}^{n,1} - u_{x}^n})  -C \l  \frac{1 + \gamma}{2}\r ({u_{y}^{n,1} - u_{y}^n} )\Bigg]. \nonumber
\end{align*}

 \item If we move $u^{n,2}$ on the right hand side  to time $t^n$, 
\begin{align*}
    u^{n,2}  =& u^n + \alpha_1 \tau u_t^{n,1} + \gamma \tau u_t^{n,2} - \gamma \tau^2 (2D-1) \frac{1+\gamma}{2} { u_t^n} + \l \frac{1 - \gamma}{2} -\alpha_1 \r \frac{u^{n,1} - u^n}{\gamma } \nonumber \\
    &  -C  \tau \l  \frac{1 + \gamma}{2}\r  (u_{x}^{n,1} - u_{x}^n) -C \tau  \l  \frac{1 + \gamma}{2}\r  (u_{y}^{n,1} - u_{y}^n).       
\end{align*}
\end{enumerate}
In both cases, for $u_x^{n,1}-u_x^n$ and $u_y^{n,1}-u_y^n$, we use \eqref{eq:uxn1} and \eqref{eq:uyn1}, respectively, passing the times $t^{n,1}$ to $t^n$.

\subsection{Third IMEX stage} \label{sm:2d_3rd}
In this case, the third stage reads:
\begin{align*}
    \frac{u^{n,3} - u^n}{\tau} =& \l 1 -\alpha_2 \r \left[ C u_x^{n,1} + C u_y^{n,1} + (2D-1)u^{n,1}  \right] + \alpha_2 \left[ C u_x^{n,2} + C u_y^{n,2} + (2D-1)u^{n,2}  \right] \nonumber \\
    &+ \beta_1 \left[ D u^{n,1}_{xx} + D u^{n,1}_{yy} \right] + \beta_2 \left[ D u^{n,2}_{xx} + D u^{n,2}_{yy}  \right] +\gamma \left[ D u^{n,3}_{xx} + D u^{n,3}_{yy}  \right]. 
\end{align*}
Following Section \ref{subsec:IMEX_gen_stage3} we get
\begin{align*}
    \frac{u^{n,3} - u^n}{\tau} =&  \l {1-\alpha_2} - \frac{(\beta_2-\alpha_2)\alpha_1}{\gamma}  \r u_t^{n,1} + \alpha_2u_t^{n,2} + \gamma u_t^{n,3} - \gamma \l 2D-1\r \l u^{n,3} - u^n\r \\ 
    & + \l \alpha_2 + \beta_1 -1 - \frac{\beta_2-\alpha_2}{\gamma} \l \frac{1-\gamma}{2}-\alpha_1\r \r \frac{u^{n,1} - u^n}{\gamma \tau} \nonumber \\
    &+ \frac{ \beta_2 - \alpha_2 }{\gamma}  \left[  \frac{u^{n,2}-u^n}{ \tau}  \right] -C  \l u_x^{n,1}-u_x^{n} \r - C   \l u_y^{n,1}-u_y^n\r.
\end{align*}
Finally, we propose two approaches:
\begin{enumerate}
\item Algebraic solution:
\begin{align*}
    u^{n,3}  = u^n + \frac{\tau}{1+\gamma\tau(2D-1)} \Bigg[ &\l {1-\alpha_2} - \frac{(\beta_2-\alpha_2)\alpha_1}{\gamma}  \r u_t^{n,1} + \alpha_2u_t^{n,2} + \gamma u_t^{n,3}  \\ 
    & \hspace{-0.1cm} + \l \alpha_2 + \beta_1 -1 - \frac{\beta_2-\alpha_2}{\gamma} \l \frac{1-\gamma}{2}-\alpha_1\r \r \frac{u^{n,1} - u^n}{\gamma \tau} \nonumber \\
    & \hspace{-1.0cm}+ \frac{ \beta_2 - \alpha_2 }{\gamma}  \l \frac{u^{n,2}-u^n}{ \tau}  \r  - C  \l u_x^{n,1}-u_x^{n} \r - C \l u_y^{n,1}-u_y^n\r  \Bigg].
\end{align*}

    \item Taking $u^{n,3}$ on the right into time $t^n$ would result in:
\begin{align*}
    u^{n,3} =& u^n + \l 1-\alpha_2 - \frac{(\beta_2-\alpha_2)\alpha_1}{\gamma}  \r \tau u_t^{n,1} + \alpha_2 \tau u_t^{n,2} + \gamma \tau u_t^{n,3}  - \gamma \tau^2 (2D-1) {u_t^n}  \\ 
    & + \l \alpha_2 + \beta_1 -1 - \frac{\beta_2-\alpha_2}{\gamma} \l \frac{1-\gamma}{2}-\alpha_1\r \r \frac{u^{n,1} - u^n}{\gamma} \nonumber \\
    &+ \frac{ \beta_2 - \alpha_2 }{\gamma}  \left[ u^{n,2}-u^n  \right] -C \tau \l u_x^{n,1}-u_x^{n} \r - C \tau \l u_y^{n,1}-u_y^n\r.   \nonumber        
\end{align*}

\end{enumerate}

\subsection{Approximations for $u^n_{xxx}$, $u^n_{yyy}$, $u^n_{yyx}$ and $u^n_{xxy}$} \label{sm:2d_4th}

Since $k=2$ for the third-order IMEX LDG scheme here considered, $\hat{u}_{xx}$ and $\hat{u}_{yy}$ are constant over each volume along the $x$ and $y$ directions, respectively. To compute $u^n_{xxx}$ for points in the boundary, we take a finite difference approximation of the second derivative over the first derivative $\hat{u}_x$. Let us assume that $(x,y) \in \Gamma(\Omega)$ and also $(x,y) \in \square_{ij}$. To shorten the notation, in the approximations below we denote $$\l \hat{u}_x \r _{i\pm a,j\pm b}  = \l \hat{u}_x\r_{i\pm a,j\pm b}  (x\pm a\Delta_i x,y\pm b\Delta_j y).$$ More precisely, the approximations are:
\begin{itemize}
    \item If possible, we take centered approximation $$\hat{u}_{xxx}(x,y) \approx \dfrac{\l \hat{u}_x\r_{i+1,j} - 2  \l \hat{u}_x\r_{i,j}  + \l \hat{u}_x\r_{i-1,j} }{(\Delta_i x)^2}.$$
    \item Otherwise, we take backward or forward approximations.
\end{itemize}
The computation of $u_{yyy}^n$ is done analogously.

Finally, we sketch the approximation of the third-order mixed derivatives with $u^n_{xxy}$. The idea is to compute the mixed derivative of $\hat{u}_x$. As an example, the approximation of $u^n_{xxy}$ at points over the south boundary is performed as follows,
\begin{itemize}
    \item If the point does not belong to volumes at the corners,
    \begin{align*}
        &\hspace{-0.45cm} \hat{u}_{xxy}(x,y)\approx \\
        & \hspace{-0.45cm} \dfrac{-\l \hat{u}_x\r_{i+1,j+2} +4 \l \hat{u}_x\r_{i+1,j+1} -3 \l \hat{u}_x\r_{i+1,j} + \l \hat{u}_x\r_{i-1,j+2} - 4\l \hat{u}_x\r_{i-1,j+1} +3 \l \hat{u}_x\r_{i-1,j}}{4\,\Delta_i x\, \Delta_j y}.   
    \end{align*}
    \item If the point belongs to the south-west corner,
    \begin{align*}
        \hat{u}_{xxy}(x,y)\approx& - \dfrac{-\l \hat{u}_x\r_{i+2,j+2} + 4 \l \hat{u}_x\r_{i+2,j+1} - 3 \l \hat{u}_x\r_{i+2,j}}{4\,\Delta_i x \, \Delta_j y} \\
        &+ 4\dfrac{-\l \hat{u}_x\r_{i+1,j+2} + \l \hat{u}_x\r_{i+1,j+1} - 3\l \hat{u}_x\r_{i+1,j}}{4 \, \Delta_i x \, \Delta_j y} \\
        &- 3\dfrac{-\l \hat{u}_x\r_{i,j+2} + 4\l \hat{u}_x\r_{i,j+1} - 3\l \hat{u}_x\r_{i,j}}{4 \, \Delta_i x \, \Delta_j y}.
    \end{align*}
    \item Otherwise (south-east corner),
        \begin{align*}
        \hat{u}_{xxy}(x,y)\approx&  3\dfrac{-\l \hat{u}_x\r_{i,j+2} + 4 \l \hat{u}_x\r_{i,j+1} - 3 \l \hat{u}_x\r_{i,j}}{4\,\Delta_i x \, \Delta_j y} \\
        &- 4\dfrac{-\l \hat{u}_x\r_{i-1,j+2} + 4\l \hat{u}_x\r_{i-1,j+1} - 3\l \hat{u}_x\r_{i-1,j}}{4 \, \Delta_i x \, \Delta_j y} \\
        &+ \dfrac{-\l \hat{u}_x\r_{i-2,j+2} + 4\l \hat{u}_x\r_{i-2,j+1} - 3\l \hat{u}_x\r_{i-2,j}}{4 \, \Delta_i x \, \Delta_j y}.
    \end{align*}
\end{itemize}
The approximations over the remaining boundaries are done in the same way. An analogous approach is used for the approximations of $u^n_{yyx}$.

\subsection{Boundary plots}  \label{sm:2d_5th}

In this section, we present some boundary plots for the numerical experiment of Section \ref{sec:exp2d}. In Figure \ref{fig:BL2D_countour}, contour plots show the final time errors in space in the cases without and with boundary treatment. Large boundary spatial errors (concerning interior errors), appearing when no boundary treatment is considered, vanish with the developed method. Finally, in Figure \ref{fig:BL2D}, some one-dimensional cuts of Figure \ref{fig:BL2D_countour} in the middle of the spatial domain are shown.

\begin{figure}[!htb]
\centering
\subfigure{\includegraphics[scale=0.42]{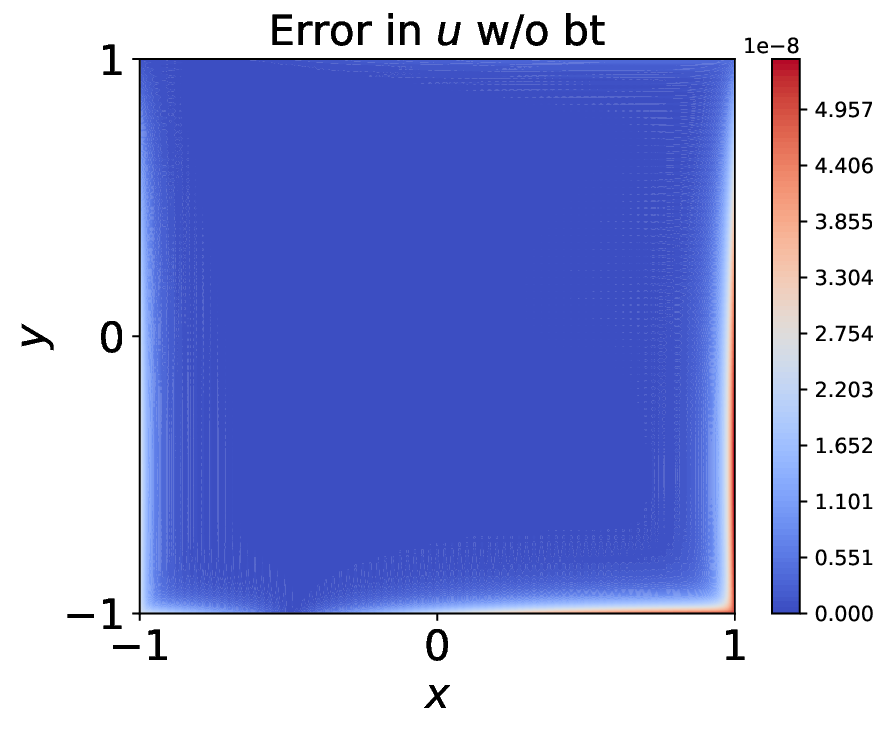}}
\subfigure{\includegraphics[scale=0.42]{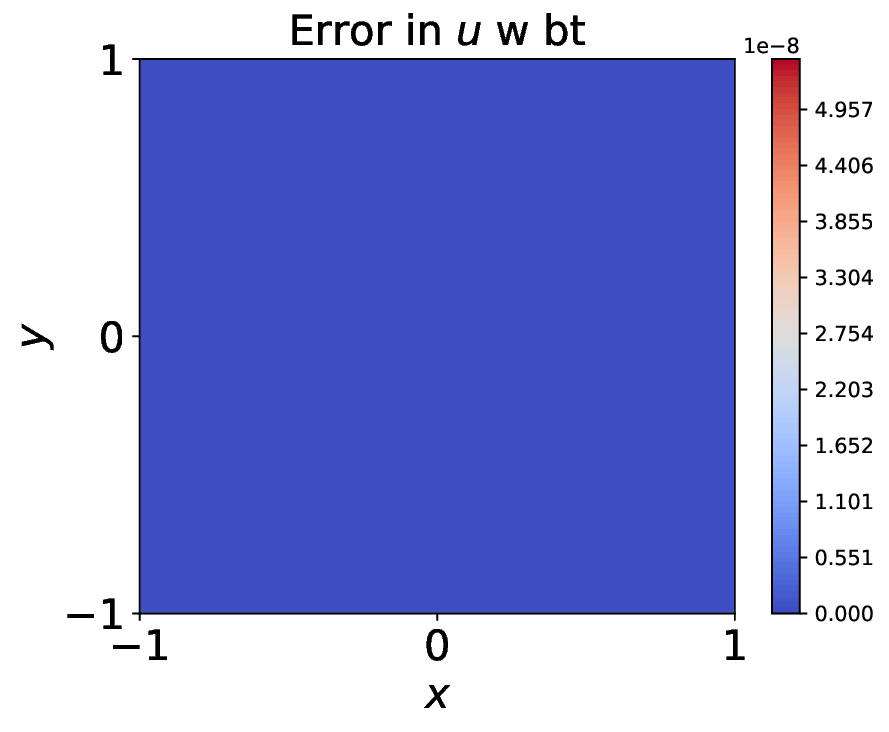}}
\caption{Contour plots for the spatial errors considering $N,M=80$ with the same data of Table \ref{tb-linear2d-sourceterms}, $t=T$.}
\label{fig:BL2D_countour}
\end{figure}

\begin{figure}[!htb]
\centering
\subfigure{\includegraphics[scale=0.43]{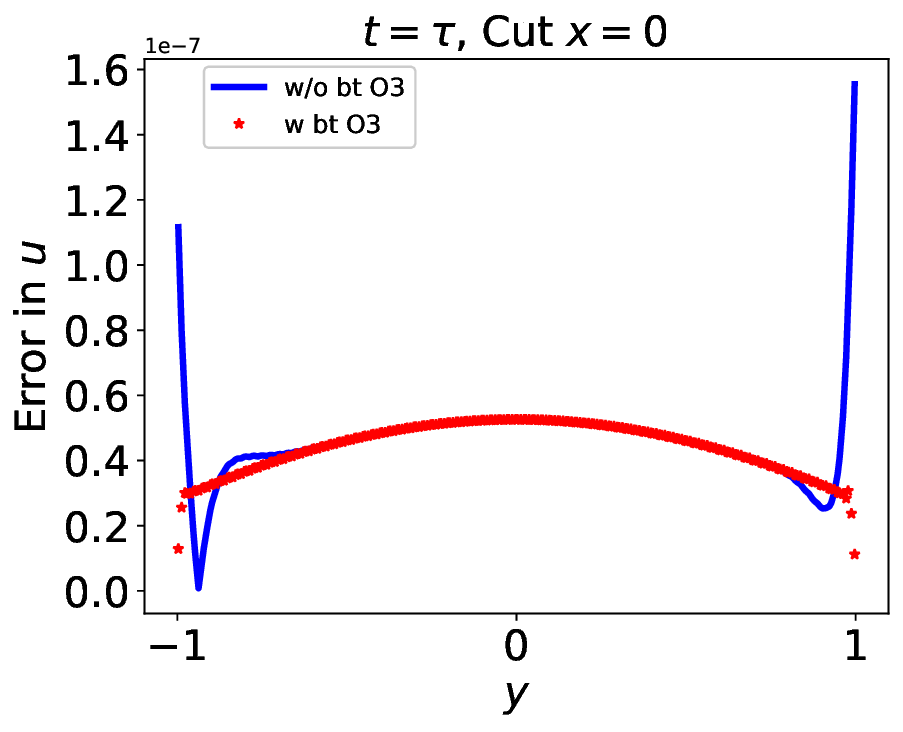}}
\subfigure{\includegraphics[scale=0.43]{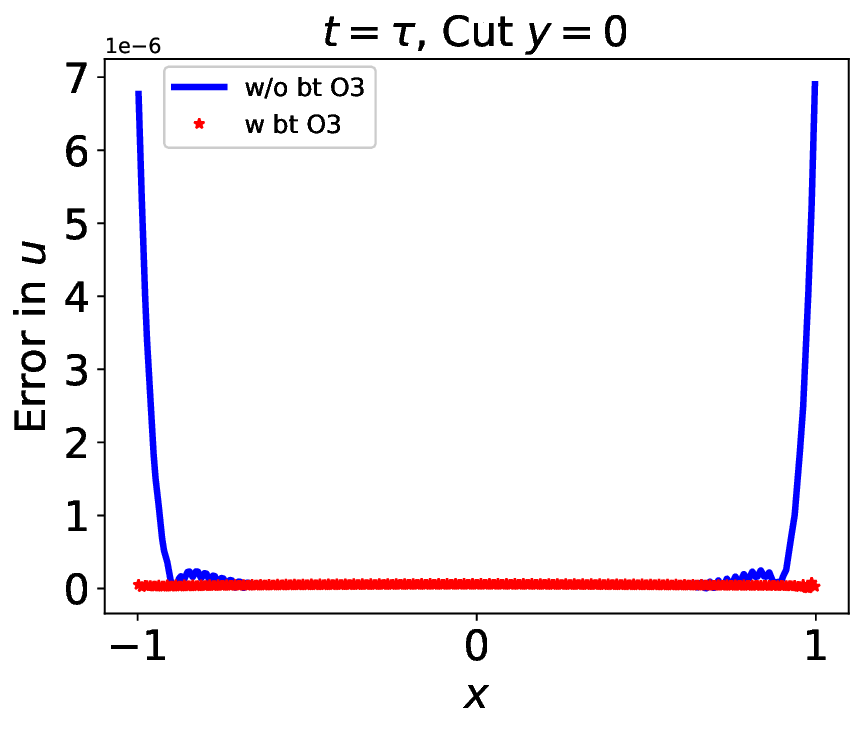}}
\subfigure{\includegraphics[scale=0.43]{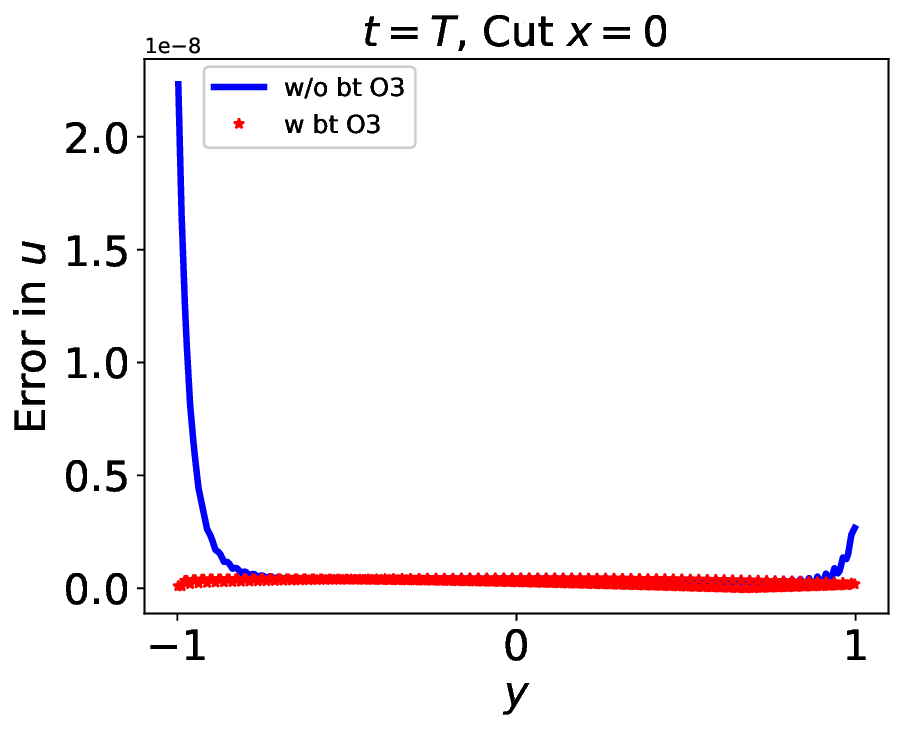}}
\subfigure{\includegraphics[scale=0.43]{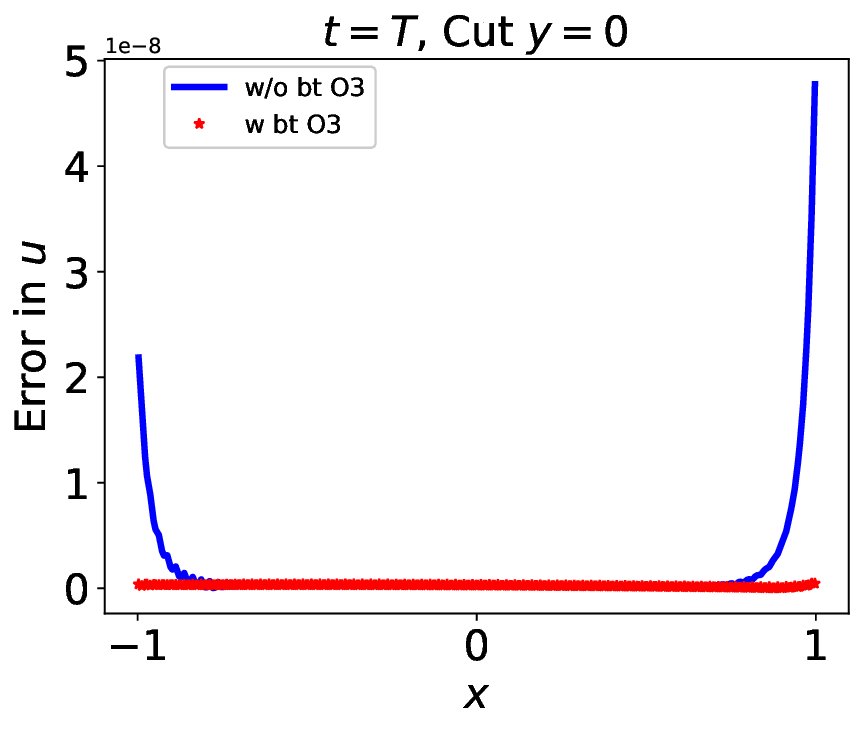}}
\caption{Spatial errors for some cuts of Figure \ref{fig:BL2D_countour}.}
\label{fig:BL2D}
\end{figure}

\end{document}